\documentclass[submission,copyright,creativecommons]{eptcs}
\providecommand{\event}{AiML 2026} 

\usepackage{aiml26}

\usepackage{iftex}
\usepackage{newtxtext,newtxmath}

\ifpdf
  \usepackage[T1]{fontenc}        
\else
  \usepackage{breakurl}           
\fi

\usepackage{amsmath}
\usepackage{mathtools}
\usepackage{tikz}
  \usetikzlibrary{arrows.meta}
\usepackage{doi}
\usepackage{multirow}

\usepackage[mathcal]{euscript}
\usepackage{mathrsfs}

\newcommand{\mc}{\mathcal}

\newcommand{\mo}{\mathfrak}
\newcommand{\B}{\mo{B}}

\newcommand{\Zin}[2]{(#1,#2)\in Z}
\newcommand{\Z}{Z}
\newcommand{\fns}{\footnotesize}
\newcommand{\ac}{R}

\newcommand{\finset}{_\mathrm{fin}}

\newcommand{\myitem}[1]{
  \renewcommand{\labelenumi}{(\theenumi) }
  \renewcommand{\theenumi}{#1}
  \item
}

\newcommand{\bisim}{\rightleftharpoons}
\newcommand{\logeq}{\leftrightsquigarrow}
\newcommand{\nlogeq}{\mathrel{\kern3pt\not\kern-3pt\leftrightsquigarrow}}
\renewcommand{\phi}{\varphi}
\renewcommand{\iff}{\quad\text{iff}\quad}
\newcommand{\und}{\underline}
\DeclareMathOperator{\st}{st}

\newcommand{\Var}{\mathrm{Var}}
\newcommand{\Con}{\mathrm{Cnst}}
\newcommand{\Pred}{\mathrm{Pred}}
\newcommand{\Prop}{\mathrm{Prop}}
\newcommand{\II}{\mathscr{I}}
\newcommand{\IL}{\mathscr{I\!\!L}}   
\newcommand{\CL}{\mathscr{C\!\!L}}   
\newcommand{\LL}{\mathcal{L}}        
\newcommand{\ML}{\mathcal{L}_{\Box\Diamond}}        
\newcommand{\IK}{\mathsf{IK}}
\newcommand{\MOC}{\mathrm{MOC}}

\newcommand{\consec}[2]{\langle #1:#2 \rangle}

\newcommand{\satisfySym}{\Vdash}
\newcommand{\satisfy}[4]{#1, #2 \satisfySym^{#4} #3}
\newcommand{\notsatisfy}[4]{#1, #2 \not\Vdash^{#4} #3}

\newcommand{\msatisfy}[3]{#1, #2 \Vdash #3}

\usepackage{hyphenat}
\title{Intuitionistic K is a Bisimulation-Invariant Fragment\\ of Intuitionistic First-Order Logic}
\author{%
  Jim de Groot\footnote{The first author was supported by Swiss National Science Foundation (SNSF) grant No.~200021\_215157.}
  \institute{University of Bern\\ Bern, Switzerland}
  \email{jim.degroot@unibe.ch}
\and
  Jo{\~a}o Marcos\footnote{The second author acknowledges the support of CNPq and UNAM / PREI DGAPA.}
  \institute{Universidade Federal de Santa Catarina\\ Florian\'opolis, Brazil}
  \email{botocudo@gmail.com}
\and
  Rodrigo Stefanes\footnote{The third author is a CNPq Scholarship Recipient, and was partially funded also by LabSEC/UFSC.}
  \institute{Universidade Federal de Santa Catarina\\ Florian\'opolis, Brazil}
  \email{rodrigoamstefanes@gmail.com}
}

\newcommand{\titlerunning}{Intuitionistic K is a Bisimulation-Invariant Fragment of Intuitionistic First-Order Logic}
\newcommand{\authorrunning}{J.~de Groot, J.~Marcos \& R.~Stefanes}

\hypersetup{
  bookmarksnumbered,
  pdftitle    = {\titlerunning},
  pdfauthor   = {\authorrunning},
  pdfsubject  = {EPTCS},               
  pdfkeywords = {keyword1, keyword2} 
}

\begin{document}

\maketitle

\begin{abstract}
  We define the notion of IK-bi\-si\-mu\-la\-tion between the relational semantics
  for the intuitionistic modal logic IK, and prove that IK arises
  as the IK-bi\-si\-mu\-la\-tion-invariant fragment of intuitionistic first-order logic.
  En route, we provide an intrinsic characterisation result of this logic by way of a
  Hennessy--Milner-style theorem and develop
  some intuitionistic first-order model theory, including 
  intuitionistic analogues of {\L}o\'{s}'s Theorem,
  elementary embeddings and countable saturation.
\end{abstract}

\section{Introduction}

Bisimulations are an important tool in the study of modal logic and computer science. In computer science, they can be used as an equivalence relation between process graphs \cite{Mil80,Park81}. In modal logic they provide a structural notion of equivalence: worlds linked through a bisimulation satisfy the same formulas. The converse result, also known as Hennessy--Milner property \cite{HenMil85}, implies that the language is powerful enough to track down structural differences between the linked worlds.
Moreover, Van Benthem's theorem, originally proved in \cite{VBen76}, provides a relative characterisation theorem that states that normal modal logic is precisely the bisimulation-invariant fragment of classical first-order logic.

Similar results have been attained for several other logics, each having an appropriate notion of bisimulation encapsulating the underlying structure of the logic. These include modal logics without negation \cite{KurRij97}, logics with negative and restorative modalities \cite{GroMarSte25},
monotone modal logic \cite{Han03},
neighbourhood logics \cite{Han09},
(bi-)intuitionistic logic \cite{Bad16, GrootPatt19, Olk13, Pat97},
modal $\mu$-calculi (within monadic second order logics) \cite{EnqSeiVen19,JanWalu95} and fragments of XPath \cite{AbrioDescFig17, FigAre15, TenBalLit10}. 

In this paper we prove a characterisation theorem for the intuitionistic modal logic $\IK$. This logic was introduced by Fischer Servi~\cite{Servi84}, and was also studied by Plotkin and Stirling~\cite{PltStir86}, Ewald~\cite{Ewa86} and Simpson~\cite{Sim94}. It is one of the myriad intuitionistic counterparts of the classical normal modal logic~$\mathsf{K}$ (see~\cite{Sim94} for an overview), and can be viewed as the collection of modal formulas whose standard translations are provable in intuitionistic first-order logic.
While existing analogues of the Van Benthem characterisation theorem for non-classical logics, such as bi-intuitionistic logic and modal extensions of positive logic, use a \emph{classical} first-order logic~\cite{KurRij97,Bad16,GroMarSte25}, our aim is to establish $\IK$ as a bisimulation-invariant fragment of \emph{intuitionistic} first-order logic.

We follow a standard path to obtain the characterisation theorem. This requires an extension of the intuitionistic ultraproduct construction and of \L{}o{\'s}'s Theorem from~\cite{Gab72,Mar79} to allow for constants. Moreover, we introduce intuitionistic analogues of elementary embeddings and $\omega$-saturation, and accompanying theorems.
On the modal side, we use the birelational models of $\IK$ to provide an intuitive notion of IK-bi\-si\-mu\-la\-tions and prove a Hennessy--Milner-style theorem. The embedding of intuitionistic first-order structures into birelational models then allows us to transfer this to the intuitionistic first-order setting.

\paragraph{Structure of the paper}
  Sections \ref{sec:ifol} and \ref{sec:ik} review basic notions and semantics for intuitionistic first-order logic and for IK.
Section~\ref{sec:biss} defines IK-bi\-si\-mu\-la\-tions and modal saturation,
  and establishes a Hennessy--Milner-style theorem.
  In Section~\ref{sec:products} we define (ultra)filter products of intuitionistic first-order structures and give an intuitionistic analogue of {\L}o{\'s}'s Theorem,
  and in Section~\ref{sec:saturation} we provide an intuitionistic
  analogue of $\omega$-saturation.
Finally, Section \ref{sec:vbtheorem} proves that intuitionistic modal logic IK is the IK-bi\-si\-mu\-la\-tion-invariant fragment of intuitionistic first-order logic.

\section{(Intuitionistic) first-order logic}
\label{sec:ifol}

  We recall some first-order logic-related material. 
  We fix throughout the text, unless otherwise stated, an arbitrary first-order signature $(\Con,\Pred)$, without function symbols, where $\Con$ and $\Pred$ are disjoint sets containing, respectively, constant symbols and predicate symbols, the latter with respective arities.
  We also fix a denumerable set $\Var$, disjoint from $\Con\cup\Pred$, of variables. 

\begin{definition}
 The \emph{language $\LL$} (for the signature $(\Con,\Pred)$) is defined by the grammar 
  \begin{equation*}
    \phi ::= P(t_1, \ldots, t_n) 
      \mid \bot 
      \mid (\phi \wedge \phi)
      \mid (\phi \vee \phi)
      \mid (\phi \to \phi)
      \mid \forall x\, \phi
      \mid \exists x\, \phi
  \end{equation*}
  where $P$, with arity $n$, ranges over $\Pred$, and where
  $t_1, \ldots, t_n$ are \emph{terms}, i.e.~elements of $\Var \cup \Con$, and $x \in \Var$.
\end{definition}

\begin{definition}
  A \emph{$\CL$-structure} (for the signature $(\Con,\Pred)$) is a pair $\mo{C} = (D, \II)$
  consisting of a nonempty set $D$ and an interpretation $\II$ that assigns to each $c \in \Con$
  an individual $\II(c) \in D$, and to each $n$-ary
  $P \in \Pred$ a relation $\II(P) \subseteq D^n$.
\end{definition}

\begin{definition}\label{def:IL-structure}
  An \emph{$\IL$-structure} (for the signature $(\Con,\Pred)$) is a tuple $\mo{M} = (W, \leq, \{ \mo{C}_w \}_{w \in W})$
  consisting of a poset $(W, \leq)$ of \emph{worlds} and for each $w \in W$ a
  $\CL$-structure $\mo{C}_w = (D_w, \II_w)$,
  for some $\Con_w \subseteq \Con$, such that
  \begin{enumerate}
    \item if $w \leq w'$, then $D_w \subseteq D_{w'}$ and
          $\II_w(P) \subseteq \II_{w'}(P)$ for any predicate symbol $P$;
    \item for each constant symbol $c \in \Con$ there is an individual
          $c^{\mo{M}} \in \bigcup_{w \in W} D_w$ such that, for each $w\in W$:
          \begin{itemize}\itemsep0pt
            \item if $c^{\mo{M}} \in D_w$, then $c \in \Con_w$ and $\II_w(c) = c^{\mo{M}}$;
            \item if $c^{\mo{M}} \notin D_w$, then $c \notin \Con_w$ and $\II_w(c)$ is not defined.
          \end{itemize}
  \end{enumerate}
  We write $D_{\mo{M}} \coloneq \bigcup \{ D_w \mid w \in W \}$ for the \emph{domain} of $\mo{M}$.

  An \emph{assignment} for an $\IL$-structure $\mo{M}$ is a mapping
  $\rho : \Var\to D_{\mo{M}}$. Given $d \in D_{\mo{M}}$ and $x\in\Var$, we write $\rho[x \coloneq d]$ for the assignment that gives the value $d$ to $x$ and agrees with $\rho$ on all other variables.
  An assignment $\rho: \Var \to D_{\mo{M}}$ can naturally be extended to a mapping $\rho^\mo{M}: \Var \cup \Con \to D_{\mo{M}}$ such that $\rho^\mo{M}(c)=c^\mo{M}$ for every $c \in \Con$.

 We interpret formulas from the set $\LL$ in $\IL$-structures under an assignment~$\rho$ via:
  \begin{alignat*}{3}
    &\satisfy{\mo{M}}{w}{P(t_1, \ldots, t_n)}{\rho}
      &&\iff (\rho^{\mo{M}}(t_1), \ldots, \rho^{\mo{M}}(t_n)) \in \II_w(P) \\
    &\satisfy{\mo{M}}{w}{\bot}{\rho}
      &&\phantom{\iff}\text{never} \\
    &\satisfy{\mo{M}}{w}{\phi \land \psi}{\rho}
      &&\iff \satisfy{\mo{M}}{w}{\phi}{\rho}
        \text{ and } \satisfy{\mo{M}}{w}{\psi}{\rho} \\
    &\satisfy{\mo{M}}{w}{\phi \vee \psi}{\rho}
      &&\iff \satisfy{\mo{M}}{w}{\phi}{\rho}
        \text{ or } \satisfy{\mo{M}}{w}{\psi}{\rho} \\
    &\satisfy{\mo{M}}{w}{\phi \to \psi}{\rho}
      &&\iff \text{for all } w' \in W,
            \text{ if } w \leq w'
            \text{ and } \satisfy{\mo{M}}{w'}{\phi}{\rho},
            \text{ then } \satisfy{\mo{M}}{w'}{\psi}{\rho} \\
    &\satisfy{\mo{M}}{w}{\forall x\, \phi}{\rho}
      &&\iff \text{for all } w' \in W,
            \text{ if } w \leq w'
            \text{ and } d \in D_{w'},
            \text{ then } \satisfy{\mo{M}}{w'}{\phi}{\rho[x \coloneq d]} \\
    &\satisfy{\mo{M}}{w}{\exists x\, \phi}{\rho}
      &&\iff \text{there exists a } d \in D_w
            \text{ such that } \satisfy{\mo{M}}{w}{\phi}{\rho[x \coloneq d]}
  \end{alignat*}
    We say that a world $w$ \emph{satisfies} a formula $\varphi \in \LL$ under the assignment $\rho$ if $\satisfy{\mo{M}}{w}{\varphi}{\rho}$. In case $\notsatisfy{\mo{M}}{w}{\varphi}{\rho}$, we say that $w$ \emph{refutes} the formula $\varphi$ under the assignment $\rho$.

  If $\phi$ has no free variables then its interpretation does not depend on
  $\rho$, and we sometimes omit reference to it.
  If $\phi$ has one free variable $x$, then satisfaction of $\phi$ depends only
  on the action of $\rho$ on~$x$, and we write $\satisfy{\mo{M}}{w}{\phi}{[x \coloneq d]}$
  to mean that $\satisfy{\mo{M}}{w}{\phi}{\rho}$ for some arbitrary assignment $\rho$ that maps $x$ to~$d$.
\end{definition}

  Note that if $\rho^{\mo{M}}(t_j) \notin D_w$, then $P(t_1, \ldots, t_n)$ is not satisfied at $w$.
  This agrees, as a matter of fact, with the treatment of atomic formulas provided in \cite[Chapter 8]{FitMend:2:2023}.
  As usual, we have persistence:

\begin{lemma}\label{mon.ifol}
  If $w \leq w'$ and $\satisfy{\mo{M}}{w}{\phi}{\rho}$, then $\satisfy{\mo{M}}{w'}{\phi}{\rho}$.
\end{lemma}

\begin{definition}
  A \emph{consecution} is a pair $\consec{\Gamma}{\Delta}$, where $\Gamma,\Delta\subseteq \LL$.
  A (\emph{finite}) \emph{subconsecution} of $\consec{\Gamma}{\Delta}$
  is a consecution $\consec{\Gamma'}{\Delta'}$ such that $\Gamma'$ and $\Delta'$
  are (finite) subsets of $\Gamma$ and $\Delta$, respectively. We write $\consec{\Gamma}{\Delta}(x_1,\ldots,x_n)$ to express that the set of free variables occurring in $\Gamma\cup \Delta$ is a subset of $x_1,\ldots,x_n$.
\end{definition}

\begin{definition}\label{def:separation-refutation}
  Let $\mo{M} = (W, \leq, \{ \mo{C}_w \}_{w \in W})$ be an $\IL$-structure.
 A world $w\in W$ and an assignment $\rho$ for~$\mo{M}$ \emph{separate} the consecution $\consec{\Gamma}{\Delta}$,
  denoted by $\satisfy{\mo{M}}{w}{\consec{\Gamma}{\Delta}}{\rho}$,
  if, under the assignment $\rho$, $w$ satisfies all formulas of $\Gamma$ and refutes all formulas of $\Delta$.
  We say that the consecution $\consec{\Gamma}{\Delta}$ is \emph{separable}
  in a set $U \subseteq W$ if there exist $w \in U$ and an assignment
  $\rho$ such that $\satisfy{\mo{M}}{w}{\consec{\Gamma}{\Delta}}{\rho}$.
  We call $\consec{\Gamma}{\Delta}$
  \emph{finitely separable} in a set $U \subseteq W$ if every finite
  subconsecution of $\consec{\Gamma}{\Delta}$ is separable in~$U$.
  A set $\Omega \subseteq \LL$ is \emph{satisfied} at $w$ under the assignment $\rho$ (notation: $\satisfy{\mo{M}}{w}{\Omega}{\rho}$), if $\satisfy{\mo{M}}{w}{\consec{\Omega}{\emptyset}}{\rho}$, and~$\Omega$ is called (finitely) satisfiable in a set $U$ if
  $\consec{\Omega}{\emptyset}$ is (finitely) separable in $U$.
  Finally, we say that $\Omega$ is \emph{refuted} at $w$ under $\rho$ if $\satisfy{\mo{M}}{w}{\consec{\emptyset}{\Omega}}{\rho}$.
\end{definition}

\section{Intuitionistic modal logic}\label{sec:ik}

  We recall the intuitionistic modal logic $\IK$ and its connection to
  intuitionistic first-order logic via the standard translation.
  Fix a denumerable set $\Prop$ of propositional letters.
  Throughout this section, we take $\Con = \emptyset$ and
  $\Pred = \Prop \cup \{ R \}$, where each symbol in $\Prop$ is taken as unary
  and $R$ is a binary predicate symbol. 
  The language $\LL$ and the $\IL$-structures in this section have $(\Con,\Pred)$ as their signature.
  
\begin{definition}
  The \emph{modal language} $\ML$ is generated by the grammar
  \begin{equation*}
    \phi ::= P
      \mid \bot
      \mid (\phi \wedge \phi)
      \mid (\phi \vee \phi)
      \mid (\phi \to \phi)
      \mid \Box\phi
      \mid \Diamond\phi
      \qquad\text{(where $P \in \Prop$)}
  \end{equation*}
\end{definition}

\begin{definition}
  The \emph{standard translation} $\Var \times \ML \to \LL$ is recursively defined by
  \begin{alignat*}{2}
    &\st(x, P) = P(x) \qquad
    \st(x, \bot) = \bot \qquad 
    &\st(x, \phi \star \psi) &= \st(x, \phi) \star \st(x, \psi) \qquad (\star \in \{ \wedge, \vee, \to \}) \\
    &\st(x, \Box\phi) = \forall y(R(x,y) \to \st(y, \phi)) \qquad
    &\st(x, \Diamond\phi) &= \exists y(R(x,y) \wedge \st(y, \phi))
  \end{alignat*}
  where $y$ is a fresh variable.
  For $\Phi \subseteq \ML$ we define $\st(x, \Phi) \coloneq \{ \st(x, \phi) \mid \phi \in \Phi \}$.
\end{definition}

  The logic $\IK$ can then be defined as a set of theorems by
  \begin{equation*}
    \IK = \{ \phi \in \ML \mid \satisfy{\mo{M}}{w}{\st(x,\phi)}{\rho} \text{ for every $\IL$-structure $\mo{M}$, world $w$ and assignment $\rho$} \}.
  \end{equation*}
  Then we can use $\IL$-structures $\mo{M} = (W, \leq, \{ \mo{C}_w \}_{w \in W})$
  to interpret formulas
  $\phi \in \ML$ at a world $w \in W$ and an individual $d \in D_w$ via
  \begin{equation*}
    \satisfy{\mo{M}}{w, d}{\phi}{} \iff \satisfy{\mo{M}}{w}{\st(x,\phi)}{[x \coloneq d]}.
  \end{equation*}
  An alternative adequate semantics for $\IK$ is given by
  birelational models~\cite[Section~8.1]{Sim94}:

\begin{definition}
  A \emph{birelational model} is a tuple $\B = (W, \leq, R, V)$ consisting of
  a poset $(W, \leq)$, a valuation~$V$ that maps propositional letters to upsets of $(W, \leq)$,
  and a binary relation $R$ on $W$
  such that:
  \begin{itemize}\itemsep0pt
    \item if $v \geq w R u$, 
          then there exists a $t \in W$ such that $vRt \geq u$, for all $u, v, w \in W$; and
    \item if $w R u \leq v$, 
          then there exists a $t \in W$ such that $w \leq t R v$, for all $u, v, w \in W$.
  \end{itemize}
  The interpretation of formulas in $\ML$ is recursively defined by
  \begin{alignat*}{3}
    &\B, w \Vdash P
    &&\iff w \in V(P) \\
    &\B, w \Vdash \bot
    &
    &\phantom{\iff}\text{never}
    \\
    &\B, w \Vdash \phi \land \psi
    &&\iff \B, w \Vdash \phi \text{ and } \B, w \Vdash \psi \\
    &\B, w \Vdash \phi \lor \psi
    &&\iff \B, w \Vdash \phi \text{ or } \B, w \Vdash \psi \\
    &\B,w \Vdash \phi \to \psi
    &&\iff\text{for every } v \in W,
          \text{ if } w \leq v
          \text{ and } \B, v \Vdash \phi,
          \text{ then } \B, v \Vdash \psi \\
    &\B, w \Vdash \Box\phi
    &&\iff\text{for every } v, u \in W,
          \text{if } w \leq v
          \text{ and } vRu,
          \text{ then } \B, u \Vdash \phi \\
    & \B, w \Vdash \Diamond\phi
    &&\iff\text{there exists a } v \in W
          \text{ such that } wRv
          \text{ and } \B, v \Vdash \phi
  \end{alignat*}
  Worlds $w_1$ and $w_2$ are called \emph{modally equivalent} (notation: $w_1 \logeq w_2$)
  if they satisfy the same formulas from $\ML$.
  We define separability, satisfaction and refutation as in
  Definition~\ref{def:separation-refutation}.
\end{definition}

  Birelational semantics enjoy the expected persistence property~\cite[Lemma~3.3.1]{Sim94}:

\begin{lemma}\label{lem:persistence}
  Let $\B = (W, \leq, \ac, V)$ be a birelational model.
  If $w, v \in W$ and $w \leq v$, then $\B, w \Vdash \phi$ implies $\B, v \Vdash \phi$, for all $\phi \in \ML$.
\end{lemma}
  
  Every $\IL$-structure gives rise to a birelational model in a truth-preserving way as follows (following Section~8.1.1 of~\cite{Sim94}).

\begin{definition}
  Let $\mo{M} = (W, \leq, \{ \mo{C}_w \}_{w \in W})$ be an $\IL$-structure.
  The \emph{induced birelational model} $\B_{\mo{M}}$ is given by 
  $(W^{\bullet}, \leq^{\bullet}, R^{\bullet}, V^{\bullet})$,
  where:
  \begin{align*}
    W^{\bullet}
      &= \{ (w, d) \mid w \in W \text{ and } d \in D_w \} 
      &(w, d) \leq^{\bullet} (w', d')
      &\iff w \leq w' \text{ and } d = d' \\
    V^{\bullet}(P)
      &= \{ (w, d) \mid d \in \II_w(P)\}
      &(w, d) R^{\bullet} (w', d')
      &\iff w = w' \text{ and } (d, d') \in \II_w(R)
  \end{align*}
\end{definition}

\begin{lemma}\label{lem:induced-birel}
  For any $\IL$-structure $\mo{M} = (W, \leq, \{ \mo{C}_w \}_{w \in W})$,
  $w \in W$, $d \in D_w$ and $\phi \in \ML$ we have
  \begin{equation*}\msatisfy{\mo{M}}{w, d}\phi \iff \msatisfy{\B_{\mo{M}}}{(w,d)}{\phi}.\end{equation*}
\end{lemma}

\section{IK-bi\-si\-mu\-la\-tions}
\label{sec:biss}

  We study IK-bi\-si\-mu\-la\-tions between birelational frames.
  These are slightly weaker than the usual definition of a (Kripke) bisimulation from
  modal logic~\cite[Section~2.2]{BlaRijVen01} because, as we shall see in
  Example~\ref{exm:no-hm}, the usual definition is too strong to prove that
  bisimilarity and modal equivalence coincide even on finite models.
  Throughout this section we use the same first-order signature as in
  Section~\ref{sec:ik}.
  Given binary relations $S$ and $R$, we write
  $S \mathrel{;} R \coloneq \{ (x, y) \mid xSz \text{ and } zRy \text{ for some } z \}$ for their sequential composition.

\begin{definition}
  Let $\B = (W, \leq, \ac, V)$ and $\B' = (W', \leq', \ac', V')$ be two
  birelational models. An \emph{IK-bi\-si\-mu\-la\-tion} between $\B$ and $\B'$ is a
  relation $\Z \subseteq W \times W'$ such that for all $\Zin{w}{w'}$:
  \begin{enumerate}
    \myitem{$\Z_p$} \label{it:sim-prop}
             $w \in V(p)$ if and only if $w' \in V'(p)$,
            for all $p\in\Prop$;
    \myitem{$\Z^1_\leq$} \label{it:sim-imp-1}
            if $w \leq v$, then there exists a $v'\in W'$ such that
            $\Zin{v}{v'}$ and $w' \leq' v'$;
    \myitem{$\Z^2_\leq$} \label{it:sim-imp-2}
            if $w' \leq' v'$, then there exists a $v\in W$ such that
            $\Zin{v}{v'}$ and $w \leq v$;
      \myitem{$\Z^1_\Box$} \label{it:sim-box-1}
              if $w \ac v$, then there exists a $v'\in W'$ such
              that $\Zin{v}{v'}$ and $w' ({\leq'} \mathrel{;} {\ac'}) v'$;
      \myitem{$\Z^2_\Box$} \label{it:sim-box-2}
              if $w' \ac' v'$, then there exists a $v\in W$ such
              that $\Zin{v}{v'}$ and $w ({\leq} \mathrel{;} {\ac}) v$;
      \myitem{$\Z^1_\Diamond$} \label{it:sim-diamond-1}
              if $wRv$, then there exist $s \in W$ and $s' \in W'$ such that
              $v \leq s$ and $w' \ac' s'$ and $\Zin{s}{s'}$;
      \myitem{$\Z^2_\Diamond$} \label{it:sim-diamond-2}
              if $w'R'v'$, then there exist $s' \in W'$ and $s \in W$ such that
              $v' \leq' s'$ and $w \ac s$ and $\Zin{s}{s'}$.
  \end{enumerate}
  We write $w \bisim w'$, and call $w$ and $w'$ \emph{bisimilar},
  if there exists an IK-bisimulation that links $w$ and $w'$.
\end{definition}

  The last four conditions can be depicted as follows, where solid and dashed arrows indicate
  universal and existential quantification, respectively:
  
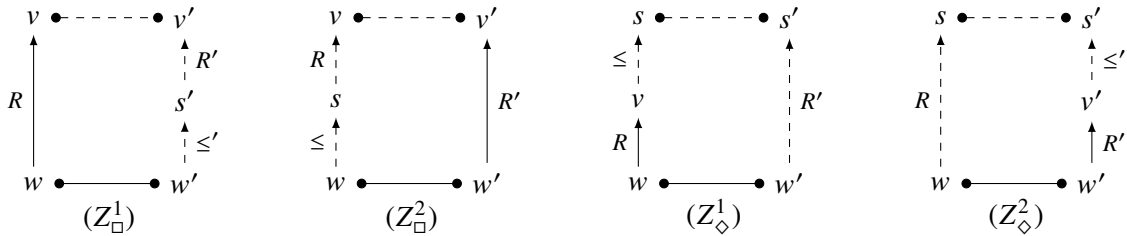
\begin{equation*}
    \begin{tikzpicture}[arrows=-latex,yscale=.73]
        \node (w)  at (0,0) {$w$};
        \node (v)  at (0,3) {$v$};
        \node (wp) at (2,0) {$w'$};
        \node (sp) at (2,1.5) {$s'$};
        \node (vp) at (2,3) {$v'$};
        \node at (1,-.6) {\eqref{it:sim-box-1}};
        \draw (w) to node[left]{\fns{$\ac$}} (v);
        \draw[Circle-Circle] (w) to (wp);
        \draw[dashed] (wp) to node[right]{\fns{$\leq'$}} (sp);
        \draw[dashed] (sp) to node[right]{\fns{$\ac'$}} (vp);
        \draw[dashed,Circle-Circle] (v) to (vp);
        \node (w)  at (4,0) {$w$};
        \node (s)  at (4,1.5) {$s$};
        \node (v)  at (4,3) {$v$};
        \node (wp) at (6,0) {$w'$};
        \node (vp) at (6,3) {$v'$};
        \node at (5,-.6) {\eqref{it:sim-box-2}};
        \draw[Circle-Circle] (w) to (wp);
        \draw (wp) to node[right]{\fns{$\ac'$}} (vp);
        \draw[dashed] (w) to node[left]{\fns{$\leq$}} (s);
        \draw[dashed] (s) to node[left]{\fns{$\ac$}} (v);
        \draw[dashed,Circle-Circle] (v) to (vp);
        \node (w)  at (8,0) {$w$};
        \node (s)  at (8,1.5) {$v$};
        \node (v)  at (8,3) {$s$};
        \node (wp) at (10,0) {$w'$};
        \node (vp) at (10,3) {$s'$};
        \node at (9,-.6) {\eqref{it:sim-diamond-1}};
        \draw[Circle-Circle] (w) to (wp);
        \draw (w) to node[left]{\fns{$\ac$}} (s);
        \draw[dashed] (s) to node[left]{\fns{$\leq$}} (v);
        \draw[dashed, Circle-Circle] (v) to (vp);
        \draw[dashed] (wp) to node[right]{\fns{$\ac'$}} (vp);
        \node (w)  at (12,0) {$w$};
        \node (v)  at (12,3) {$s$};
        \node (wp) at (14,0) {$w'$};
        \node (sp)  at (14,1.5) {$v'$};
        \node (vp) at (14,3) {$s'$};
        \node at (13,-.6) {\eqref{it:sim-diamond-2}};
        \draw[Circle-Circle] (w) to (wp);
        \draw (wp) to node[right]{\fns{$\ac'$}} (sp);
        \draw[dashed] (sp) to node[right]{\fns{$\leq'$}} (vp);
        \draw[dashed, Circle-Circle] (v) to (vp);
        \draw[dashed] (w) to node[left]{\fns{$\ac$}} (v);
    \end{tikzpicture}
  \end{equation*}

Note that bisimilarity implies model equivalence, that is:

\begin{proposition}\label{prop:adeq}
  Let $\B = (W, \leq, \ac, V)$ and $\B' = (W', \leq', \ac', V')$ be two
  birelational models. Then $w \bisim w'$ implies $w \logeq w'$,
  for all $w \in W$ and $w' \in W'$.
\end{proposition}
\begin{proof}
  We prove something stronger, namely that for every formula~$\varphi$ and for any pair of worlds $w, w'$ such that $w\bisim w'$ we have $\B,w\Vdash \varphi$ iff $\B',w'\Vdash \varphi$.  
  This may be checked by structural induction on~$\varphi$.
  We showcase the inductive steps for $\Diamond$ and $\Box$.  
  Assume, then, for a certain formula $\psi$, the following induction hypothesis,  P[$\psi$]: $\B,w\Vdash \psi$ iff $\B',w'\Vdash \psi$, for any pair of worlds $w, w'$ such that $w\bisim w'$.
  Accordingly, in what follows, take arbitrary worlds $w, w'$ such that $w \bisim w'$.  
  
  \medskip\noindent
  \textit{$\triangleright\;$ Case for $\phi = \Diamond\psi$.} Suppose $\B, w \Vdash \phi$.
  In view of the assumption that $w \bisim w'$, let $Z \subseteq W \times W'$ be an $\IK$ bisimulation
  between $\B$ and $\B'$ such that $(w, w') \in Z$.
  From $\B, w \Vdash \Diamond\psi$ we may obtain a world $v \in W$
  such that $w \ac v$ and $\B, v \Vdash \psi$. Using~\eqref{it:sim-diamond-1}
  we obtain $s \in W$ and $s' \in W'$ such that $v \leq s$, $w'\ac's'$ and $(s,s') \in Z$. The induction hypothesis and Lemma~\ref{lem:persistence}
  then imply $\B', s' \Vdash \psi$, hence $\B', w' \Vdash \Diamond\psi$.
  The converse (namely, if $\B', w' \Vdash \Diamond\psi$ then $\B, w \Vdash \Diamond\psi$)
  can be proven similarly.

  \medskip\noindent
  \textit{$\triangleright\;$ Case for $\phi = \Box\psi$.} Suppose $\B, w \Vdash \phi$.
  Let $Z \subseteq W \times W'$ be an $\IK$ bisimulation between $\B$ and $\B'$ such that $(w, w') \in Z$.
  Let $t', v' \in W'$ be such that $w' \leq' t' \ac' v'$.
  Then we can use~\eqref{it:sim-imp-2} and~\eqref{it:sim-box-2} to find
  $t, s, v \in W$ such that $w \leq t \leq s \ac v$ and $(v, v') \in Z$.
  This implies $w \leq s$, so by Lemma~\ref{lem:persistence} we have $\B, s \Vdash \Box\psi$.
  Hence $\B, v \Vdash \psi$, and by the induction hypothesis
  $\B', v' \Vdash \psi$. This entails that $\B', w' \Vdash \Box\psi$.
  The converse is proven similarly.

  This concludes the inductive proof of P[$\varphi$], for every~$\varphi$.

  Take now worlds $w, w'$ such that $w \bisim w'$.  Given an arbitrary formula~$\varphi$, we may use P[$\varphi$] to conclude that $w \logeq w'$.
\end{proof}

\newcommand{\larr}[4]{\draw (#1) to node[#4]{\footnotesize{$#3$}} (#2)}
\begin{example}\label{exm:no-hm}
  Consider the two birelational models drawn below,
  where the circled worlds indicate the valuation of a proposition letter $P$, and
  the intuitionistic accessibility relation is the reflexive closure of the depicted $\leq$-arrows:
  
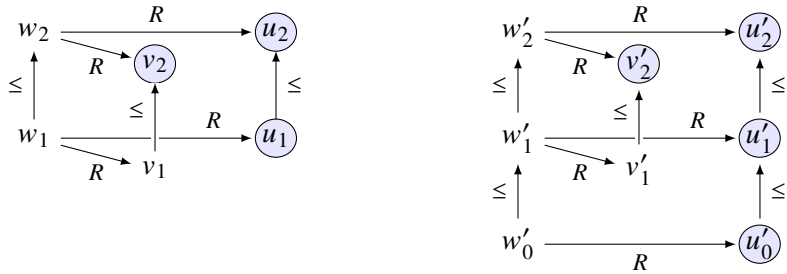
\begin{equation*}
    \begin{tikzpicture}[xscale=1.6,yscale=1.4,arrows=-latex]
        \foreach \i in {(1,.7),(2,0),(2,1),(5,.7),(6,-1),(6,0),(6,1)}{
        \draw[fill=blue!10] \i ellipse(1.7mm and 1.8mm);
        }
        \node (w1) at (0,0) {$w_1$};
        \node (w2) at (0,1) {$w_2$};
        \node (v1) at (1,-.3) {$v_1$};
        \node (v2) at (1,.7) {$v_2$};
        \node (u1) at (2,0) {$u_1$};
        \node (u2) at (2,1) {$u_2$};
        \larr{w1}{v1}{R}{below};
        \larr{w1}{u1}{R}{above,pos=.8};
        \larr{w2}{v2}{R}{below};
        \larr{w2}{u2}{R}{above};
        \larr{w1}{w2}{\leq}{left};
        \draw[white,-,line width=3pt] (v1) to (v2);
        \larr{v1}{v2}{\leq}{left,pos=.6};
        \larr{u1}{u2}{\leq}{right};
        \node (w0) at (4,-1) {$w_0'$};
        \node (w1) at (4,0) {$w_1'$};
        \node (w2) at (4,1) {$w_2'$};
        \node (v1) at (5,-.3) {$v_1'$};
        \node (v2) at (5,.7) {$v_2'$};
        \node (u0) at (6,-1) {$u_0'$};
        \node (u1) at (6,0) {$u_1'$};
        \node (u2) at (6,1) {$u_2'$};
        \larr{w0}{u0}{R}{below};
        \larr{w1}{v1}{R}{below};
        \larr{w1}{u1}{R}{above,pos=.8};
        \larr{w2}{v2}{R}{below};
        \larr{w2}{u2}{R}{above};
        \larr{w1}{w2}{\leq}{left};
        \draw[white,-,line width=3pt] (v1) to (v2);
        \larr{v1}{v2}{\leq}{left,pos=.6};
        \larr{u1}{u2}{\leq}{right};
        \larr{w0}{w1}{\leq}{left};
        \larr{u0}{u1}{\leq}{right};
    \end{tikzpicture}
  \end{equation*}
  The following relation is an IK-bi\-si\-mu\-la\-tion:
  \begin{equation*}
    \Z = \big\{ (w_1, w_0'), (w_1, w_1'), (w_2, w_2'), (v_1, v_1') \big\}
        \cup \big\{ (x, y) \mid x \in \{ v_2, u_1, u_2 \} \text{ and } y \in \{ v_2', u_0', u_1', u_2' \} \big\}.
  \end{equation*}
  In particular, $w_1$ and $w'_0$ are modally equivalent. However, there is no Kripke bisimulation between $w_1$ and $w'_0$, since $w_1Rv_1$ cannot be mirrored at $w_0'$. 
\end{example}

\begin{definition}
  We define an IK-bi\-si\-mu\-la\-tion between two $\IL$-structures $\mo{M}$ and $\mo{M}'$
  as an IK-bi\-si\-mu\-la\-tion between the induced birelational models $\B_{\mo{M}}$ and $\B_{\mo{M}'}$.
\end{definition}

  Next, we give a notion of saturation that encompasses image-finite birelational
  models, and prove that on the class of saturated birelational models
  bisimilarity coincides with logical equivalence.

\begin{definition}\label{def:mod-sat}
  Let $\B = (W, \leq, \ac, V)$ be a birelational model.
  A subset $X \subseteq W$ is called
  \begin{enumerate}
    \item \emph{positively saturated} if every set $\Phi \subseteq \ML$ that is finitely satisfiable in $X$ is also satisfiable in $X$;
    \item \emph{saturated} if every consecution $\consec{\Phi}{\Psi}$ that is finitely separable in $X$ is also separable in $X$. 
  \end{enumerate}
  We call $\B$
  \emph{modally saturated} if for each $w \in W$ the set
  $R[w] \coloneq \{ v \in W \mid wRv \}$ is positively saturated,
  and the sets ${\uparrow}w \coloneq \{ v \in W \mid w \leq v \}$
  and $R_{\uparrow}[w] \coloneq \{ v \in W \mid w ({\leq} \mathrel{;} R) v \}$
  are both saturated.
\end{definition}

\begin{theorem}\label{thm:hm}
  Let $\B = (W, \leq, \ac, V)$ and $\B' = (W', \leq', \ac', V')$ be two modally
  saturated birelational models. Then bisimilarity and modal equivalence
  coincide between these models.
\end{theorem}
\begin{proof}
  We prove that the relation $Z$ of modal equivalence  is an IK-bi\-si\-mu\-la\-tion.
  It clearly satisfies~\eqref{it:sim-prop},
  and~\eqref{it:sim-imp-1} and~\eqref{it:sim-imp-2} are similar
  to~\cite[Theorem~21]{Pat97}.
  We show that $Z$ satisfies~\eqref{it:sim-box-1} and~\eqref{it:sim-diamond-1},
  the remaining conditions being similar.
  For~\eqref{it:sim-box-1}, let $\Zin{w}{w'}$ and $w \ac v$
  and suppose towards a contradiction that there exists no
  $v' \in \ac'_{\uparrow}[w'] \coloneq \{ x' \in W' \mid w'({\leq'} \mathrel{;} \ac')x' \}$
  such that $\Zin{v}{v'}$.
  Then for each such $v'$ either
  \begin{itemize}\itemsep0pt
    \item there exists a formula $\phi_{v'}$ that is true at $v$ but not at $v'$; or
    \item there exists a formula $\psi_{v'}$ that is true at $v'$ but not at $v$.
  \end{itemize}
  For each $v' \in \ac'_{\uparrow}[w']$,
  pick such $\phi_{v'}$ or $\psi_{v'}$, and collect them in
  sets $\Phi$ and $\Psi$, respectively.
  Then there is no world in $\ac'_{\uparrow}[w']$ that
  separates $\consec{\Phi}{\Psi}$.
  Since $\ac'_{\uparrow}[w']$ is saturated, this implies that there exist
  finite subsets $\Phi' \subseteq \Phi$ and $\Psi' \subseteq \Psi$ such that
  no world in $\ac'_{\uparrow}[w']$ separates $\consec{\Phi'}{\Psi'}$.
  Set $\phi \coloneq \bigwedge \Phi'$ and $\psi \coloneq \bigvee \Psi'$.
  It follows from frame condition (2) on $\B'$ 
  that $w' ({\leq'} \mathrel{;} \ac' \mathrel{;} {\leq'}) s'$ implies
  $w' ({\leq'} \mathrel{;} {\ac'}) s'$, that is, $R'_\uparrow[w']$ is upward-closed under~$\leq'$.
  Therefore, each $v'$ such that
  $w' ({\leq'} \mathrel{;} {\ac'}) v'$ satisfies $\phi \to \psi$.
  Furthermore, we have $v \not\Vdash \phi \to \psi$, so that ultimately $w \not\Vdash \Box(\phi \to \psi)$ while $w' \Vdash \Box(\phi \to \psi)$, a contradiction.

  Next, to see that $\Z$ satisfies~\eqref{it:sim-diamond-1},
  let $\Zin{w}{w'}$ and $w \ac v$ and suppose towards a contradiction
  that there exist no $s \in W$ and $s' \in W'$ such that $v \leq s$ and $w' \ac' s'$
  and $\Zin{s}{s'}$.
  Note that, by assumption, in particular, ${\uparrow}v$ is saturated. So, similarly to the previous case,
  for each $s' \in R'[w']$ we can find formulas~$\phi_{s'}$ and $\psi_{s'}$
  such that $s'$ satisfies $\phi_{s'}$ and refutes $\psi_{s'}$,
  and each $s \geq v$ either refutes $\phi_{s'}$ or satisfies $\psi_{s'}$.
  Let $\xi_{s'} \coloneq \phi_{s'} \to \psi_{s'}$. Then $v \Vdash \xi_{s'}$ and $s' \not\Vdash \xi_{s'}$.
    Now let $\Xi \coloneq \{ \xi_{s'} \mid s' \in R'[w'] \}$.
    Then no world in $R'[w']$ satisfies $\Xi$.
    Since $R'[w']$ is positively saturated by assumption, it follows that there
    exists a finite $\Xi' \subseteq \Xi$ such that each $s' \in R'[w']$
    refutes some formula in $\Xi'$.
    Let $\xi \coloneq \bigwedge \Xi'$. Then $v \Vdash \xi$, so $w \Vdash \Diamond\xi$ while $w' \not\Vdash \Diamond\xi$,
    a contradiction.
\end{proof}

\section{Filter products of $\IL$-structures}
\label{sec:products}

  We define an intuitionistic analogue of the ultrafilter product of
  classical 
  first-order structures, and use this to prove compactness of the logic.
  We will also use it in Section~\ref{sec:saturation} to prove that every
  $\IL$-structure can be elementarily embedded in an $\omega$-saturated one.
  The notion of (ultra)filter product of $\IL$-structures
  ---without constant symbols---
  may be found in~\cite{Gab72} and~\cite[Chapter~III]{Mar79}.
  To improve overall legibility, details of some of our constructions and proofs of some of our results were shifted to Appendix~\ref{app:filter-product}.
  The remaining work in this section mirrors the development of classical
  model theory (see e.g.~\cite{ChaKei73}).

\begin{definition}
  Let $I$ be a set, and for each $i \in I$ let
  $\mo{M}_i = (W_i, \leq_i, \{ \mo{C}_{i,w} \}_{w \in W_i})$
  be an $\IL$-structure, where $\mo{C}_{i, w} = (D_{i, w}, \II_{i, w})$.
  Let $F$ be a filter on $I$. Then:
  \begin{enumerate}
    \item Take $W$ to be the reduced product $\prod_{i \in I}^F W_i$, 
          namely the elements in $\prod_{i \in I} W_i$ modulo the equivalence relation
          given by $\alpha \sim \beta$ iff $\{ i \in I \mid \alpha(i) = \beta(i) \} \in F$.
          The equivalence class of $\alpha$ is denoted by~$\alpha_F$.
    \item Define the relation $\leq$ on $W$ by $\alpha_F \leq \beta_F$ if and only if
          $\{ i \in I \mid \alpha(i) \leq_i \beta(i) \} \in F$
    \item For each $i \in I$, let $D_i \coloneq D_{\mo{M}_i} = \bigcup_{w \in W_i} D_{i,w}$.
          Let $D = \prod_{i \in I}^F D_i$ be the reduced product of the $D_i$, namely
          the elements in $\prod_{i \in I} D_i$ modulo the equivalence relation
          $\xi \approx \eta$ iff $\{ i \in I \mid \xi(i) = \eta(i) \} \in F$,
          and denote the equivalence classes by $\xi_F$.
          Define the domain $D_{\alpha_F}$ at $\alpha_F \in W$ by
          \begin{equation*}
            D_{\alpha_F} \coloneq \{ \xi_F \in D \mid \{ i \in I \mid \xi(i) \in D_{i,\alpha(i)} \} \in F \}.
          \end{equation*}
    \item For $c \in \Con$,  define
          $\tilde{c} : I \to \bigcup_{i \in I} D_i$ by $\tilde{c}(i) = c^{\mo{M}_i}$,
          and let $\II_{\alpha_F}(c) \coloneq \tilde{c}_F$ if $\tilde{c}_F \in D_{\alpha_F}$
          and leave the interpretation of $c$ undefined otherwise.
    \item Finally, for each $n$-ary predicate $P \in \Pred$, define its interpretation by
          \begin{equation*}
            (\xi_F^1, \ldots, \xi_F^n) \in \II_{\alpha_F}(P)
              \iff \{ i \in I \mid (\xi^1(i), \ldots, \xi^n(i)) \in \II_{i, \alpha(i)}(P) \} \in F.
        \end{equation*}
  \end{enumerate}
  For each $\alpha_F \in W$ we get a $\CL$-structure $\mo{C}_{\alpha_F} = (D_{\alpha_F}, \II_{\alpha_F})$.
  The \emph{filter product} of the $\mo{M}_i$ is given by
  \begin{equation*}
    \prod_{i \in I}^F \mo{M}_i = (W, \leq, \{ \mo{C}_{\alpha_F} \}_{\alpha_F \in W}).
  \end{equation*}
\end{definition}

\begin{proposition}\label{prop:filter-product}
  Let $I$ be a set, and for each $i \in I$ let $\mo{M}_i$ be an
  $\IL$-structure. Let $F$ be a filter on $I$.
  Then the filter product $\prod_{i \in I}^F \mo{M}_i$ is an $\IL$-structure.
\end{proposition}

\begin{definition}
  Let $I$ be a set, let $\mo{M}_i$ be an $\IL$-structure, for each $i \in I$, and let $F$ be a filter over $I$.
  If $F$ is an ultrafilter, then the filter product $\prod_{i \in I}^F \mo{M}_i$ is called an
  \emph{ultraproduct}.
  Finally, if $\mo{M}_i = \mo{M}$ for all $i \in I$,
  then the ultraproduct is called an \emph{ultrapower} of $\mo{M}$. 
\end{definition}

  The next goal is the following analogue of {\L}o{\'s}'s Theorem.
  This requires the following notion of a product of assignments.

\begin{definition}\label{def:prod-assignment}
  Suppose $\rho_i$ is an assignment for $\mo{M}_i$, for each $i \in I$.
  For each $x \in \Var$, define
  $\rho(x) : I \to \bigcup_{i \in I} D_i : i \mapsto \rho_i(x)$.
  Then the \emph{product assignment} $\rho_F$
  maps $x$ to the equivalence class of $\rho(x)$,
  i.e.~$\rho_F(x) \coloneq \rho(x)_F$.
\end{definition}

  We obtain the following analogue of {\L}o{\'s}'s Theorem,
  a proof of which can be found in Appendix~\ref{app:filter-product}.

\begin{theorem}\label{thm:Los}
  Let $I$ be a set, and for each $i \in I$ let $\mo{M}_i$ be an $\IL$-structure
  and $\rho_i$ an assignment for~$\mo{M}_i$.
  Let $F$ be an ultrafilter on $I$,
  $\mo{M} \coloneq \prod_{i \in I}^F \mo{M}_i$ be the filter product,
  and $\rho_F$ be
  the product assignment for~$\mo{M}$.
  Then for any world $\alpha_F \in W$ and any formula $\phi(x_1, \ldots, x_n)$
  such that $\rho_F(x_1), \ldots, \rho_F(x_n) \in D_{\alpha_F}$,
  \begin{equation*}
    \satisfy{\mo{M}}{\alpha_F}{\phi}{\rho_F}
      \iff \{ i \in I \mid \satisfy{\mo{M}_i}{\alpha(i)}{\phi}{\rho_i} \} \in F.
  \end{equation*}
\end{theorem}

  As an application of {\L}o{\'s}'s Theorem, we prove that every $\IL$-structure
  can be elementarily embedded in any of its ultrapowers.
  We begin by defining an intuitionistic analogue of elementary embeddings.

\begin{definition}
  Let $\mo{M} = (W, \leq, \{ \mo{C}_{w} \}_{w \in W})$ and
  $\mo{M}' = (W', \leq', \{ \mo{C}_{w'} \}_{w' \in W'})$ be two $\IL$-structures.
  An \textit{elementary embedding} of $\mo{M}$ into $\mo{M'}$ is a pair
  $(\epsilon, \eta)$ of injective functions
  $\epsilon : W \to W'$ and $\eta : D_{\mo{M}} \to D_{\mo{M}'}$ such that
  for every $\phi \in \mc{L}$, $w \in W$ and assignment $\rho$:
  \begin{equation*}
    \satisfy{\mo{M}}{w}{\phi}{\rho}
      \iff \satisfy{\mo{M}'}{\epsilon(w)}{\phi}{\eta \circ \rho}.
  \end{equation*}
\end{definition}

\begin{definition}
  Let $\mo{M} = (W, \leq, \{ \mo{C}_w \}_{w \in W})$ be an $\IL$-structure,
  $I$ be a set, $F$ be an ultrafilter over $I$, and write
  $\mo{M}^* = (W^*, \leq^*, \{ \mo{C}_{\alpha_F} \}_{\alpha_F \in W^*})$
  for the ultrapower of $\mo{M}$ modulo $F$.
  As in Definition \ref{def:IL-structure}, let $D_{\mo{M}^*} \coloneq \bigcup_{\alpha_F \in W^*} D_{\alpha_F}$.
  For each $w \in W$, let $w^* \in W^* \coloneq \prod_{i \in I}^F W$ be the equivalence class of the
  constant function $I \to W : i \mapsto w$
  and for each $d \in D_{\mo{M}}$ let $d^* \in D_{\mo{M}^*} = \prod_{i \in I}^F D_{\mo{M}}$
  be the equivalence class of the constant map $I \to D_{\mo{M}} : i \mapsto d$.
  The \emph{natural embedding} of $\mo{M}$ into $\mo{M}^*$ is given by
  \begin{equation*}
    \epsilon : W \to W^* : w \mapsto w^*
    \quad\text{and}\quad
    \eta : D_{\mo{M}} \to D_{\mo{M}^*} : d \mapsto d^*.
  \end{equation*}
\end{definition}

\begin{proposition}\label{prop:ultra-embedding}
  Let $\mo{M}$ be an $\IL$-structure, $I$ be a set and $F$ be an ultrafilter over $I$.
  The natural embedding of $\mo{M}$ into $\mo{M}^*$ is an elementary embedding.
\end{proposition}
\begin{proof}
  The maps $w \mapsto w^*$ and $d \mapsto d^*$ are clearly injective.
  Let $\rho$ be any assignment for $\mo{M}$, and let
  $\rho^*$ be the product assignment on $\mo{M}^*$ induced by $\rho$
  (as in Definition~\ref{def:prod-assignment}).
  Then $\rho^* = \eta \circ \rho$.
  Furthermore, for any $w \in W$ we have $w^*(i) = w$,
  so we get
  \begin{align*}
    \satisfy{\mo{M}^*}{w^*}{\phi}{\eta \circ \rho}
      &\iff \satisfy{\mo{M}^*}{w^*}{\phi}{\rho^*}
      &\text{(because $\rho^* = \eta \circ \rho$)} \\
      &\iff \{ i \in I \mid \satisfy{\mo{M}}{w^*(i)}{\phi}{\rho} \} \in F
      &\text{(Theorem~\ref{thm:Los})} \\
      &\iff \{ i \in I \mid \satisfy{\mo{M}}{w}{\phi}{\rho} \} \in F
      &\text{(because $w^*(i) = w$)} \\
      &\iff \satisfy{\mo{M}}{w}{\phi}{\rho}
  \end{align*}
  The last ``iff'' follows from the fact that
  $\{ i \in I \mid \satisfy{\mo{M}}{w}{\phi}{\rho} \}$ is either
  empty (if $\notsatisfy{\mo{M}}{w}{\phi}{\rho}$)
  or equal to~$I$ (if $\satisfy{\mo{M}}{w}{\phi}{\rho}$),
  so it is in $F$ if and only if the latter is the case.
\end{proof}

  Finally, still mirroring the classical case,
  we use \L{}o{\'s}'s Theorem~\ref{thm:Los} to prove compactness of the logic.

\begin{definition}
  Given $\Gamma,\Delta \subseteq \LL$, we say that $\Delta$
  \emph{is a consequence of} $\Gamma$, and write $\Gamma\models \Delta$,
  if for every $\IL$-structure $\mo{M}$, 
  there is no $w \in W$ and no assignment $\rho$ 
  that separate $\consec{\Gamma}{\Delta}$ in $\mo{M}$.
  We say that two formulas $\varphi$ and $\psi$ are \emph{logically equivalent} if $\{\varphi\}\models \{\psi\}$ and $\{\psi\}\models \{\varphi\}$. 
\end{definition}

\begin{theorem}\label{compactness}
  Let $\consec{\Gamma}{\Delta}(x_1,\ldots,x_n)$
  be a consecution in $\LL$ and let $I$ be the set of all of its finite sub\-con\-se\-cu\-tions, i.e.~$
    I = \{ \consec{\Gamma\finset}{\Delta\finset} \mid \Gamma\finset \subseteq \Gamma \text{ and } \Delta\finset \subseteq \Delta \text{ are finite} \}.
  $
  For each $i \in I$, let $\mo{M}_i = (W_i, {\leq_i,} \{ \mo{C}_{i, w} \}_{w \in W_i})$
  be an $\IL$-structure, $w_i \in W_i$
  and $\rho_i$ an assignment such that $\satisfy{\mo{M}_i}{w_i}{i}{\rho_i}$. 
  Then there exists an ultrafilter $F$ over $I$ such that
  $\satisfy{\prod_{i \in I}^F \mo{M}_i}{ \{ (i, w_i) \mid i \in I \}_F}{\consec{\Gamma}{\Delta}}{\rho_F}$,
  where $\rho_F$ is the product assignment of the $\rho_i$.
\end{theorem}
\begin{proof}
  For each pair $(\gamma,\delta)$ such that
  $\gamma \in \Gamma$ and $\delta \in \Delta$,
  define 
  \begin{equation*}
    \overline{(\gamma,\delta)} = \{ \consec{\Gamma\finset}{\Delta\finset} \in I \mid \gamma \in \Gamma\finset \text{ and } \delta \in \Delta\finset \}
  \quad\text{and}\quad
    E = \{\overline{(\gamma,\delta)} \mid \gamma \in \Gamma \text{ and } \delta \in \Delta \}.
  \end{equation*}
  Then $E$ has the finite intersection property because for any given $\overline{(\gamma_1,\delta_1)},\ldots, \overline{(\gamma_n,\delta_n)} \in E$ we have
  \begin{equation*}
    \consec{ \{\gamma_1,\ldots,\gamma_n\}}{\{\delta_1,\ldots,\delta_n\} }  \in {\overline{(\gamma_1,\delta_1)}}\cap\cdots\cap {\overline{(\gamma_n,\delta_n)}}.
  \end{equation*}
  By the ultrafilter theorem \cite[Theorem~4.1.3]{ChaKei73} there exists an 
  ultrafilter $F$ on $I$ such that $E \subseteq F$.
  Now consider an arbitrary pair $(\gamma, \delta)$ such that
  $\gamma \in \Gamma$ and $\delta \in \Delta$, and $i= \consec{\Gamma\finset}{\Delta\finset} \in I$. 
  If $i \in \overline{(\gamma,\delta)}$, then $\gamma \in \Gamma\finset, \delta \in \Delta\finset$
  and hence $\satisfy{\mo{M}_i}{w_i}{\gamma}{\rho_i}$ and $\notsatisfy{\mo{M}_i}{w_i}{\delta}{\rho_i}$. Therefore,
  $\overline{(\gamma,\delta)} \subseteq \{ i \in I \mid \satisfy{\mo{M}_i}{w_i}{\gamma}{\rho_i}\}$
  and $\overline{(\gamma,\delta)} \subseteq \{ i \in I \mid \notsatisfy{\mo{M}_i}{w_i}{\delta}{\rho_i} \}$,
  and since $\overline{(\gamma, \delta)} \in F$ by construction we get
  \begin{equation*}
    \{ i \in I \mid \satisfy{\mo{M}_i}{w_i}{\gamma}{\rho_i}\} \in F
    \quad\text{and}\quad
    \{ i \in I \mid \notsatisfy{\mo{M}_i}{w_i}{\delta}{\rho_i}\} \in F.
  \end{equation*}
  Because $F$ is an ultrafilter we get
  $\{ i \in I \mid \satisfy{\mo{M}_i}{w_i}{\delta}{\rho_i}\} \not\in F$,
  so Theorem~\ref{thm:Los} entails
  \begin{equation*}
    \satisfy{\prod_{i \in I}^F\mo{M}_i}{\{(i,w_i) \mid i \in I \}_F}{\gamma}{\rho_F}
    \quad\text{and}\quad
    \notsatisfy{\prod_{i \in I}^F\mo{M}_i}{\{(i,w_i) \mid i \in I \}_F}{\delta}{\rho_F},
  \end{equation*}
  where $\{ (i, w_i) \mid i \in I \}_F$ denotes the equivalence class
  of the function $I \to \bigcup_i W_i : i \mapsto w_i$.
  Note that $\rho_i$ does not depend on $\gamma$ or $\delta$.
  Therefore $\{(i,w_i) \mid i \in I \}_F$ and $\rho_F$ separate the consecution
  $\consec{\Gamma}{\Delta}$ in $\prod_{i \in I}^F\mo{M}_i$.
\end{proof}

\begin{corollary}\label{cor.compactness}
  If $\Gamma \models \Delta$, then there exist finite subsets
  $\Gamma\finset \subseteq \Gamma$ and $\Delta\finset \subseteq \Delta$
  such that $\Gamma\finset\models \Delta\finset$.
\end{corollary}
\begin{proof}
  Let $\Gamma \models \Delta$ and suppose towards a contradiction that for all
  finite $\Gamma\finset \subseteq \Gamma$ and $\Delta\finset \subseteq \Delta$,
  we have $\Gamma\finset \not\models \Delta\finset$.
  Let $I$ be the set of finite subconsecutions of $\consec{\Gamma}{\Delta}$.
  Then for every $i \in I$ there exists an $\IL$-structure
  $\mo{M}_i = (W_i, \leq_i, \{ \mo{C}_{i,w} \}_{w \in W_i})$, a world $w_i \in W_i$ and an assignment $\rho_i$
  such that $\satisfy{\mo{M}_i}{w_i}{i}{\rho_i}$.
  Therefore Theorem~\ref{compactness} implies that there exists an
  ultrafilter $F$ over $I$ such that $\satisfy{\prod_{i \in I}^F \mo{M}_i}{\{(i,w_i) \mid i \in I\}_F}{\consec{\Gamma}{\Delta}}{\rho_F}$,
  where $\rho_F$ is the product assignment.
  As a consequence $\Gamma \not\models \Delta$, a contradiction.
\end{proof}

\section{Countable saturation in the intuitionistic setting}\label{sec:saturation}

  A final ingredient we need in the proof of the characterisation theorem for
  $\IK$ is a way to turn an $\IL$-structure into a modally saturated one.
  More precisely, we wish to elementarily embed any $\IL$-structure
  $\mo{M}$ into an $\IL$-structure $\mo{M}^*$ whose induced birelational
  model $\B_{\mo{M}^*}$ is modally saturated in the sense of
  Definition~\ref{def:mod-sat}.
  Keeping in sync with the classical case, we do so using an intuitionistic
  variation of the notion of $\omega$-saturation which is tailored towards
  our needs as an intermediary between $\IL$-structures and modally saturated
  birelational models.

\begin{definition}
  Given an $\IL$-structure $\mo{M} = (W, \leq, \{ \mo{C}_w \}_{w \in W})$
  and $A\subseteq\bigcup_{w\in W}D_w$,
  we write $\LL_A$ for the extension of $\LL$ with constants
  $\{ \und{a} \mid a \in A \}$. We define $(\mo{M})_A$ as the $\IL_A$-structure that extends the $\IL$-structure $\mo{M}$ by interpreting
  each of the new constants $\und{a}$ as $a$.
\end{definition}

\begin{definition}
  Let $\mo{M} = (W, \leq, \{ \mo{C}_w \}_{w \in W})$ be an $\IL$-structure.
  \begin{enumerate}
    \item A set $U \subseteq W$ is called \emph{positively $\omega$-saturated}
          if for all finite $A \subseteq \bigcap_{w \in U} D_w$
          and every set $\Gamma(x) \subseteq \LL_A$
          with one free variable $x$,
          if $\Gamma(x)$ is finitely satisfiable in $U$,
          then it is satisfiable in $U$.
    \item The set $U \subseteq W$ is called \emph{$\omega$-saturated}
          if for all finite $A \subseteq \bigcap_{w \in U} D_w$
          and all sets $\Gamma(x), \Delta(x) \subseteq \LL_A$
          with one free variable $x$,
          if $\consec{\Gamma(x)}{\Delta(x)}$ is finitely separable in $U$,
          then it is separable in $U$.
  \end{enumerate}
  The $\IL$-structure $\mo{M}$ is called \emph{$\omega$-saturated} if, for every $w \in W$, the set
  $\{ w \}$ is positively $\omega$-saturated
  and ${\uparrow}w$ is $\omega$-saturated.
\end{definition}

  We verify that every $\omega$-saturated $\IL$-structure gives rise to a modally saturated birelational model.

\begin{proposition}\label{prop:ind_sat}
  Let $\mo{M} = (W, \leq, \{ \mo{C}_w \}_{w \in W})$ be an
  $\omega$-saturated $\IL$-structure. Then the induced birelational model
  $\B_{\mo{M}} = (W^{\bullet}, \leq^{\bullet}, R^{\bullet}, V^{\bullet})$ is
  modally saturated.
\end{proposition}
\begin{proof}
  Let $(w, d) \in W^{\bullet}$ be an arbitrary world.
  One may check, similarly to~\cite[Theorem~2.65]{BlaRijVen01}, that $R^{\bullet}[(w, d)]$ is
  positively saturated. We verify the two remaining saturation conditions
  from Definition~\ref{def:mod-sat}.
  
  \medskip\noindent
  \textit{${\uparrow}(w,d)$ is saturated.}
    Let $\Phi, \Psi \subseteq \ML$ be sets of formulas that are finitely separable in ${\uparrow}(w,d)$. Consider the consecution $\consec{\st(x, \Phi)}{\st(x, \Psi)}$.
    Then each of its finite subconsecutions is of the form
    $\consec{\st(x, \Phi')}{\st(x, \Psi')}$ for some finite $\Phi' \subseteq \Phi$
    and $\Psi' \subseteq \Psi$.
    By assumption, we can find some $(v, d) \in {\uparrow}(w, d)$ that
    satisfies $\Phi'$ and refutes $\Psi'$, which implies
    $\satisfy{\mo{M}}{v}{\bigwedge\st(x,\Phi')}{[x \coloneq d]}$
    and $\notsatisfy{\mo{M}}{v}{\bigvee\st(x,\Psi')}{[x \coloneq d]}$.
    Therefore $\consec{\st(x, \Phi')}{\st(x, \Psi')}$ is finitely separable
    in ${\uparrow}w$, and since $\mo{M}$ is assumed to be $\omega$-saturated
    it must also be separable in some world $u \in {\uparrow}w$.
    This entails that $(u, d) \in {\uparrow}(w, d)$ separates $\consec{\Phi}{\Psi}$, as desired.

  \medskip\noindent
  \textit{$R^\bullet_{\uparrow}[(w,d)]$ is saturated.}
    Let $\Phi, \Psi \subseteq \ML$ and suppose that for any finite
    $\Phi' \subseteq \Phi$ and $\Psi' \subseteq \Psi$ there exists a world
    $(v, e) \in R^\bullet_{\uparrow}[(w, d)]$ 
    separating $\consec{\Phi'}{\Psi'}$.
    Note that $(v, e) \in R^{\bullet}_{\uparrow}[(w, d)]$ iff
    $(w, d) \leq^{\bullet} (v, d) R^{\bullet} (v, e)$,
    iff $w \leq v$ and $(d, e) \in \II_v(R)$.
    Define
    \begin{equation*}
      \Gamma \coloneq \{R\und{d}x\} \cup \st(x, \Phi)
      \quad\text{and}\quad \Delta \coloneq \st(x, \Psi).
    \end{equation*}
    We claim that $\consec{\Gamma}{\Delta}$ is finitely separable in
    ${\uparrow}w$ on the $\IL$-structure $\mo{M}$.
    To see this, let $\consec{\Gamma'}{\Delta'}$ be a finite subconsecution of $\consec{\Gamma}{\Delta}$,
    and let $\Phi' \coloneq \{ \phi \in \Phi \mid \st(x, \phi) \in \Gamma' \}$
    and $\Psi' \coloneq \{ \psi \in \Psi \mid \st(x, \psi) \in \Delta' \}$.
     Then $\Phi'$ and $\Psi'$ are finite, so by assumption we can find some
    $(v, e) \in R^\bullet_{\uparrow}[(w, d)]$ that separates $\consec{\Phi'}{\Psi'}$.
    This gives $\satisfy{\mo{M}}{v}{\consec{\Gamma'}{\Delta'}}{[x \coloneq e]}$, since $(v, e) \in R^\bullet_{\uparrow}[(w, d)]$ implies $\satisfy{\mo{M}}{v}{R\und{d}x}{[x\coloneq e]}$.
    As we have seen, $(v, e) \in R^\bullet_{\uparrow}[(w, d)]$
    implies $w \leq v$,
    so $\consec{\Gamma}{\Delta}$ is finitely separable in ${\uparrow}w$.
    Since $\mo{M}$ is $\omega$-saturated, we can find some $u \in {\uparrow}w$
    and $b \in D_u$ such that $\satisfy{\mo{M}}{u}{\consec{\Gamma}{\Delta}}{[x\coloneq b]}$.
    From this it follows that $(u, b) \in W^{\bullet}$ is an element of $R^\bullet_{\uparrow}[(w,d)]$ that separates $\consec{\Phi}{\Psi}$.
    We conclude that $R^\bullet_{\uparrow}[(w, d)]$ is saturated, as desired.
\end{proof}

  Recall that a countably incomplete ultrafilter on a set $I$ is an ultrafilter
  that is not closed under countably infinite
  intersections~\cite[Propositions~4.3.4 and~4.3.5]{ChaKei73}.
  Modifying~\cite[Theorem~6.1.1]{ChaKei73} we have:

\begin{proposition}\label{prop:ultra-sat}
  Let $\LL$ be a language over a countable signature, let $I$ be a set and $\mo{M}_i$ an $\IL$-structure for each $i \in I$.
  If $F$ is a countably incomplete ultrafilter on $I$,
  then $\prod_{i \in I}^F\mo{M}_i$ is $\omega$-saturated.
\end{proposition}
\begin{proof}
\textit{Part 1: for any world $\alpha_F$ in $\prod_{i \in I}^F\mo{M}_i$,
        the set $\{ \alpha_F \}$ is positively $\omega$-saturated.}
  Let $\alpha_F$ be an element of the ultraproduct $\prod_{i \in I}^F \mo{M}_i$.
  Let $A \subseteq D_{\alpha_F}$ be a finite set of individuals of $\alpha_F$
  and let $\IL_A$ be the extension of the language with constants
  $\underline{a}$ for each $a_F \in A$.
  Let $A_i = \{ a(i) \mid a_F \in A \}$.
  Then we have
  \begin{equation*}
    \Big(\prod_{i \in I}^F\mo{M}_i\Big)_A = \prod_{i \in I}^F((\mo{M}_i)_{A_i}).
  \end{equation*}
  Let $\Gamma(x)=\{ \gamma_1(x), \gamma_2(x), \ldots \}$ be a
  subset of $\LL_A$ that is
  finitely satisfiable in $\{ \alpha_F \}$. 
  Since $\Gamma(x)$ is finitely satisfiable in $\{ \alpha_F \}$,
  for each $n \in \mathbb{N}$ we can find some $d \in D_{\alpha_F}$ such that
  $\satisfy{\prod_{i \in I}^F(\mo{M}_i)_{A_i}}
           {\alpha_F}
           {\gamma_1(x) \wedge \cdots \wedge \gamma_n(x)}
           {[x \coloneq d]}$.
  Hence for each $n \in \mathbb{N}$,
  \begin{equation*} 
    \satisfy{\prod_{i \in I}^F((\mo{M}_i)_{A_i})}{\alpha_F}{\exists x(\gamma_1(x)\land\cdots\land\gamma_n(x))}{}
  \end{equation*}
  
  Using the fact that $F$ is countably incomplete,
  we can find a descending chain
  $I = I_0 \supseteq I_1 \supseteq I_2 \supseteq \cdots$ of sets in $F$
  such that $\bigcap_{n \in \mathbb{N}} I_n = \emptyset$.
  Let $X_0 = I$, and for each $n \in \mathbb{N}_{>0}$ define
  \begin{equation*}
    X_n = I_n \cap \{i \in I \mid \satisfy{(\mo{M}_i)_{A_i}}{\alpha(i)}{\exists x(\gamma_1(x)\land \cdots \wedge \gamma_n(x))}{}\}
  \end{equation*}
  By \L{}o{\'s}'s Theorem, the right part of the intersection belongs
  to $F$ and hence $X_n \in F$.
  Moreover, we have $X_n \supseteq X_{n+1}$ for all $n \in \mathbb{N}$,
  and $\bigcap_{n \in\mathbb{N}} X_n = \emptyset$. 
  As a consequence, for each $i \in I$ there exists an $n(i) \in \mathbb{N}$
  such that $n(i)$ is the greatest natural number with $i \in X_{n(i)}$,
  so there exists an assignment $\rho_i$ for $\mo{M}_i$ such that
  \begin{equation*}
    \satisfy{(\mo{M}_i)_{A_i}}{\alpha(i)}{\gamma_1(x)\land\cdots\land\gamma_{n(i)}(x)}{\rho_i}
  \end{equation*}

  Now let $n \in \mathbb{N}_{>0}$ and $i \in X_n$. 
  Then $n \leq n(i)$ and hence
  $\satisfy{(\mo{M}_i)_{A_i}}{\alpha(i)}{\gamma_n(x)}{\rho_i}$. Since $X_n \in F$, we have 
  \begin{equation*}
    X_n \subseteq \{ i \in I \mid \satisfy{(\mo{M}_i)_{A_i}}{\alpha(i)}{\gamma_n(x)}{\rho_i} \}
        \in F
  \end{equation*}
  so, by \L{}o{\'s}'s Theorem~\ref{thm:Los}, we have
  $\satisfy{\prod_{i \in I}^F((\mo{M}_i)_{A_i})}{\alpha_F}{\gamma_n(x)}{\rho_F}$,
  where $\rho$ is the product assignment.
  Since this holds for any $n \in \mathbb{N}_{>0}$,
  it follows that $\Gamma(x)$ is satisfiable in $\{\alpha_F\}$.

\medskip\noindent
\textit{Part 2: for any world $\alpha_F$ in $\prod_{i \in I}^F\mo{M}_i$,
        the set ${\uparrow}\alpha_F$ is $\omega$-saturated.}
  Let $\alpha_F$ be an element of the ultraproduct $\prod_{i \in I}^F \mo{M}_i$.
  Let $A \subseteq \bigcap_{\beta_F \in {\uparrow}\alpha_F}D_{\beta_F} = D_{\alpha_F}$ be a finite set of individuals of $\alpha_F$
  and let $\IL_A$ be the extension of the language with constants
  $\underline{a}$ for each $a_F \in A$.
  Let $A_i = \{ a(i) \mid a_F \in A \}$.
  As before, we have $(\prod_{i \in I}^F\mo{M}_i)_A = \prod_{i \in I}^F((\mo{M}_i)_{A_i})$.

  Let $\Gamma(x),\Delta(x) \subseteq \LL_A$ be countable sets such that
  $\prod_{i \in I}^F((\mo{M}_i)_{A_i})$ finitely separates $\consec{\Gamma(x)}{\Delta(x)}$
  in the set ${\uparrow}\alpha_F$.
  Then we can write $\Gamma(x) = \{ \gamma_1(x), \gamma_2(x), \ldots \}$
  and $\Delta(x) = \{ \delta_1(x), \delta_2(x), \ldots \}$.
  By hypothesis, for each $n \in \mathbb{N}$ there exists an $\alpha_F \leq \beta^n_F$
  and an assignment $\rho_F^n$ such that
  \begin{equation*}
    \satisfy{\prod_{i \in I}^F((\mo{M}_i)_{A_i})}
            {\beta^n_F}
            {\consec{\{\gamma_1(x),\ldots,\gamma_n(x)\}}{\{\delta_1(x),\ldots,\delta_n(x)\}}}
            {\rho_F^n}
  \end{equation*}
  and hence, since $\alpha_F \leq \beta^n_F$,
  \begin{equation*}
    \notsatisfy{\prod_{i \in I}^F((\mo{M}_i)_{A_i})}
               {\alpha_F}
               {\forall x( (\gamma_1(x) \wedge \cdots \wedge \gamma_n(x)) \to (\delta_1(x) \vee \cdots \vee \delta_n(x)))}
               {}
  \end{equation*}

  Again, the fact that $F$ is countably incomplete yields an infinite
  descending chain $I = I_0 \supseteq I_1 \supseteq I_2 \supseteq \cdots$
  of elements in $F$ such that $\bigcap_{n \in \mathbb{N}} I_n = \emptyset$.
  Let $X_0 = I$, and for each $n \in \mathbb{N}_{>0}$ define
  \begin{equation*}
    X_n = I_n \cap \{i  \in I \mid \notsatisfy{(\mo{M}_i)_{A_i}}{\alpha(i)}{\forall x( (\gamma_1(x) \wedge \cdots \wedge \gamma_n(x)) \to (\delta_1(x) \vee \cdots \vee \delta_n(x)))}{}\}
  \end{equation*}
  By \L{}o{\'s}'s Theorem~\ref{thm:Los}, the right part of the intersection
  belongs to $F$ and hence $X_n \in F$. Moreover, $X_n \supseteq X_{n+1}$
  and $\bigcap_{n \in\mathbb{N}} X_n = \emptyset$. 
  Hence for each $i \in I$ there exists an $n(i) \in \mathbb{N}$ such that $n(i)$
  is the greatest natural number such that $i \in X_{n(i)}$, so that there exists an 
  assignment $\rho_i$ for $\mo{M}_i$ and some $a_i \geq \alpha(i)$ in $W_i$
  such that
  \begin{equation*}
    \notsatisfy{(\mo{M}_i)_{A_i}}{a_i}{(\gamma_1(x) \wedge \cdots \wedge \gamma_{n(i)}(x)) \to (\delta_1(x) \vee \cdots \vee \delta_{n(i)}(x))}{\rho_i}
  \end{equation*}
  Therefore, for every $i \in I$, there exists some $b_i \geq_i a_i$ such that 
  \begin{equation*}
    \satisfy{(\mo{M}_i)_{A_i}}{b_i}{\consec{\{\gamma_1(x),\ldots,\gamma_{n(i)}(x)\}}{\{\delta_1(x),\ldots,\delta_{n(i)}(x)\}}}{\rho_i}.
  \end{equation*}

  For each $n \in \mathbb{N}_{>0}$ and $i \in X_n$ we have $n \leq n(i)$, hence
  $\satisfy{(\mo{M}_i)_{A_i}}{b_i}{\consec{\{\gamma_n(x)\}}{\{\delta_n(x)\}}}{\rho_i}$.
  Since $X_n \in F$, 
  \begin{equation*}
    X_n \subseteq \{ i \in I \mid \satisfy{(\mo{M}_i)_{A_i}}{b_i}{\consec{\{\gamma_n(x)\}}{\{\delta_n(x)\}}}{\rho_i}\} \in F
  \end{equation*}
  so using \L{}o{\'s}'s Theorem~\ref{thm:Los} we obtain
  $\satisfy{\prod_{i \in I}^F((\mo{M}_i)_{A_i})}{\beta_F}{\consec{\{\gamma_n(x)\}}{\{\delta_n(x)\}}}{\rho_F}$, where $\rho$ is the product assignment and $\beta(i)=b_i$ for every $i \in I$.
  Since this holds for every $n \in \mathbb{N}_{>0}$,
  we may conclude that 
  $\satisfy{\prod_{i \in I}^F((\mo{M}_i)_{A_i})}{\beta_F}{\consec{\Gamma(x)}{\Delta(x)}}{\rho_F}$
  so that $\consec{\Gamma(x)}{\Delta(x)}$ is separable in ${\uparrow}\alpha_F$.
\end{proof}

\begin{theorem}\label{thm:ultra-embedding}
  Every $\IL$-structure for a countable signature can be elementarily embedded in an $\omega$-saturated $\IL$-structure.
\end{theorem}
\begin{proof}
  Combine the fact that countably incomplete ultrafilters exist (see Proposition~4.3.5 of~\cite{ChaKei73})
  with Propositions~\ref{prop:ultra-sat} and~\ref{prop:ultra-embedding}.
\end{proof}

\section{Van Benthem-style characterisation theorem}
\label{sec:vbtheorem}

  By now we have developed all ingredients needed to prove the characterisation
  theorem for $\IK$. 
  Throughout this section we use the same first-order signature as in Section~\ref{sec:ik}.

\begin{definition}
  A formula $\phi(x) \in \LL$ with one free variable $x$ is said to be
  \emph{invariant under IK-bi\-si\-mu\-la\-tions} if for every pair of $\IL$-structures
  $\mo{M} = (W, \leq, \{ \mo{C}_w \}_{w \in W})$ and
  $\mo{M}' = (W', \leq', \{ \mo{C}_{w'} \}_{w' \in W'})$,
  worlds $w \in W$ and $w'\in W'$, and individuals $d \in D_w$ and $d'\in D_{w'}$, if there exists an IK-bi\-si\-mu\-la\-tion $B$ between $\B_{\mo{M}}$ and~$\B_{\mo{M}'}$
  linking $(w, d)$ and $(w', d')$, then
  \begin{equation*}
    \satisfy{\mo{M}}{w}{\phi(x)}{[x \coloneq d]}
      \iff \satisfy{\mo{M}'}{w'}{\phi(x)}{[x \coloneq d']}.
  \end{equation*}
\end{definition}

\begin{theorem}
  A formula $\alpha(x) \in \LL$ with one free variable $x$ is equivalent to
  the translation of a formula in $\ML$ if and only if it is invariant under
  IK-bi\-si\-mu\-la\-tions.
\end{theorem}
\begin{proof}
  Suppose $\alpha(x)$ is equivalent to $\st(x,\psi)$ for some $\psi \in \ML$.
  Let $\mo{M}$ and $\mo{M}'$ be two $\IL$-structures and
  $\Z$ be an IK-bi\-si\-mu\-la\-tion between $\B_{\mo{M}}$ and $\B_{\mo{M}'}$ linking
  $(w, d)$ to $(w', d')$. Then
  \begin{align*}
    \satisfy{\mo{M}}{w}{\alpha(x)}{[x \coloneq d]}
      &\iff \satisfy{\mo{M}}{w}{\st(x,\psi)}{[x \coloneq d]} \\ 
      &\iff \msatisfy{\B_{\mo{M}}}{(w,d)}{\psi}
      &\text{(by definition and Lemma~\ref{lem:induced-birel})} \\
      &\iff \msatisfy{\B_{\mo{M'}}}{(w',d')}{\psi}
      &\text{(Proposition~\ref{prop:adeq})} \\
      &\iff \satisfy{\mo{M}'}{w'}{\st(x,\psi)}{[x \coloneq d']} 
      &\text{(by definition and Lemma~\ref{lem:induced-birel})} \\
      &\iff  \satisfy{\mo{M}'}{w'}{\alpha(x)}{[x \coloneq d']}
  \end{align*}
  So, $\alpha$ is invariant under IK-bi\-si\-mu\-la\-tions.

  For the converse, assume that $\alpha(x)$ is invariant under IK-bi\-si\-mu\-la\-tions.
  Consider the set
  \begin{equation*}
    \MOC(\alpha) = \{ \st(x,\phi) \mid \phi \in \ML
                     \text{ and } \alpha(x) \models \st(x, \phi) \}.
  \end{equation*}
  It suffices to show $\MOC(\alpha) \models \alpha(x)$, because then
  compactness (Corollary~\ref{cor.compactness}) yields a finite
  subset $\{ \st(x, \psi_1), \ldots, \st(x, \psi_n) \} \subseteq \MOC(\alpha)$
  that has $\alpha(x)$ as a consequence, from which it follows that $\alpha(x)$ is
  equivalent to $\st(x, \psi_1 \wedge \cdots \wedge \psi_n)$.
  
  Assume $\satisfy{\mo{M}}{w}{\MOC(\alpha)}{[x \coloneq d]}$.
  We need to show that $\satisfy{\mo{M}}{w}{\alpha(x)}{[x \coloneq d]}$.
  Let
  \begin{equation*}
    \Gamma(x) \coloneq \{ \st(x,\phi) \mid \satisfy{\mo{M}}{w}{\st(x,\phi)}{[x \coloneq d]} \}
    \quad\text{and}\quad
    \Delta(x) \coloneq \{ \st(x,\psi) \mid \notsatisfy{\mo{M}}{w}{\st(x,\psi)}{[x \coloneq d]} \}.
  \end{equation*}
  Suppose towards a contradiction that
  $\Gamma(x) \cup \{ \alpha(x) \} \models \Delta(x)$.
  Then Corollary~\ref{cor.compactness} yields finite
  $\Gamma\finset(x) \subseteq \Gamma(x)$
  and $\Delta\finset(x) \subseteq \Delta(x)$ such that
  $\Gamma\finset(x) \cup \{ \alpha(x) \} \models \Delta\finset(x)$.
  This implies $\alpha(x) \models \bigwedge \Gamma\finset(x) \to \bigvee \Delta\finset(x)$,
  hence $\bigwedge \Gamma\finset(x) \to \bigvee \Delta\finset(x) \in \MOC(\alpha)$.
  By assumption, $\satisfy{\mo{M}}{w}{\MOC(\alpha)}{[x \coloneq d]}$,
  so $\satisfy{\mo{M}}{w}{\bigwedge \Gamma\finset(x) \to \bigvee \Delta\finset(x)}{[x \coloneq d]}$.
  But then the fact that $\satisfy{\mo{M}}{w}{\st(x,\psi)}{[x\coloneq d]}$ for all
  $\st(x, \psi) \in \Gamma\finset(x)$ contradicts the fact that
  $\notsatisfy{\mo{M}}{w}{\st(x, \chi)}{[x \coloneq d]}$ for all
  $\st(x, \chi) \in \Delta\finset(x)$.

  So we have $\Gamma(x) \cup \{ \alpha(x) \} \not\models \Delta(x)$,
  hence there must exist an $\IL$-structure $\mo{N} = (V, \leq, \{ \mo{C}_v \}_{v \in V})$,
  a world $v \in V$ and an individual $e \in D_v$ such that 
  $\mo{N}, v$ separates $\consec{\Gamma(x)\cup\{\alpha(x)\}}{\Delta(x)}$ under $[x \coloneq e]$.
  Then by construction we have
  $\satisfy{\mo{M}}{w}{\st(x,\varphi)}{[x \coloneq d]}$ if and only if
  $\satisfy{\mo{N}}{v}{\st(x,\varphi)}{[x \coloneq e]}$ for all $\phi \in \ML$.
  Proposition~\ref{prop:ultra-embedding}
  gives elementary embeddings of $\mo{M}$ and $\mo{N}$ into
  $\omega$-saturated $\IL$-structures $\mo{M}^*$ and $\mo{N}^*$,
  respectively.
  Hence, for every $\phi \in \ML$
  \begin{align*}
    \satisfy{\mo{M}^*}{w^*}{\st(x, \phi)}{[x \coloneq d^*]}
      &\iff \satisfy{\mo{M}}{w}{\st(x, \phi)}{[x \coloneq d]} \\
      &\iff \satisfy{\mo{N}}{v}{\st(x, \phi)}{[x \coloneq e]}
      \iff \satisfy{\mo{N}^*}{v^*}{\st(x,\varphi)}{[x \coloneq e^*]}.
  \end{align*}
  By Proposition~\ref{prop:ind_sat},
  $\B_{\mo{M}^*}$ and $\B_{\mo{N}^*}$ are modally saturated birelational models,
  so that Lemma~\ref{lem:induced-birel} entails
  \begin{equation*}
    \msatisfy{\B_{\mo{M}^*}}{(w^*, d^*)}{\phi}
      \iff \msatisfy{\B_{\mo{N}^*}}{(v^*, e^*)}{\phi}
  \end{equation*}
  for all $\phi \in \ML$. Hence, by Theorem~\ref{thm:hm}
  there exists an IK-bi\-si\-mu\-la\-tion $\Z$ between $\B_{\mo{M}^*}$ and $\B_{\mo{N}^*}$
  linking $(w^*, d^*)$ and $(v^*, e^*)$.
  Since $\alpha$ is invariant under IK-bi\-si\-mu\-la\-tions,
  we conclude that 
  \begin{equation*}
    \satisfy{\mo{M}}{w}{\alpha(x)}{[x \coloneq d]}
      \iff \satisfy{\mo{M}^*}{w^*}{\alpha(x)}{[x \coloneq d^*]}
      \iff \satisfy{\mo{N}^*}{v^*}{\alpha(x)}{[x \coloneq e^*]}
      \iff \satisfy{\mo{N}}{v}{\alpha(x)}{[x \coloneq e]}.
  \end{equation*}
  By construction the right-hand side holds, so that ultimately
  $\satisfy{\mo{M}}{w}{\alpha(x)}{[x \coloneq d]}$, as desired.
\end{proof}

\section{Conclusions}

Motivated by Van Benthem's celebrated characterisation theorem, we introduced a precise notion of IK-bi\-si\-mu\-la\-tion between birelational models and proved that the modal logic $\IK$ corresponds exactly to the IK-bi\-si\-mu\-la\-tion-invariant fragment of the intuitionistic first-order logic with one binary predicate and a unary predicate for each propositional letter.
En route to this result, we developed intuitionistic counterparts of classical model theory machinery, including an intuitionistic version of \L{}o{\'s}'s Theorem, a compactness theorem and notions of elementary embeddings and countable saturation.

Within the model-theoretic framework, a natural next step is to generalise our results and constructions so as to allow for function symbols in the signature, as well as to weaken the restrictions made on the interpretation of terms, and to expand the theory of saturated models beyond the countable case. Regarding the modal logical framework, obtaining a characterisation theorem for IK relative to classical first-order logic may also be of interest, as well as to further explore the relation between Kripke bisimulations and IK-bi\-si\-mu\-la\-tions to check whether they yield the same invariance notion.

Finally, it would be worth investigating whether the notion of IK-bisimulation developed here is robust to changes in the chosen intuitionistic modal logic (e.g.~$\mathsf{CK}$~\cite{BelPaiRit01,Pai03,MenPai05}, $\mathsf{WK}$~\cite{Wij90,WijNer05}, $\mathsf{FS}$ \cite{WolterZakharyaschev1999}, or other constructive variants), or whether each non-classical variant of $\mathsf{K}$ needs its own tailored bisimulation.

\appendix
\section{The filter product construction}\label{app:filter-product}

  We elaborate on the definition of the filter product.
  Proposition~\ref{prop:filter-product} follows from Lemmas~\ref{lem:B1} to~\ref{lem:B3}.

  \bigskip\noindent
  \textbf{The setup}
  Let $I$ be a set, and for each $i \in I$ let
  $\mo{M}_i = (W_i, \leq_i, \{ \mo{C}_{i,w} \}_{w \in W_i})$ be an $\IL$-structure,
  where $\mo{C}_{i, w} = (D_{i, w}, \II_{i, w})$.
  Let $F$ be a filter on $I$.
  In what follows, we aim to define the \emph{filter product}
  $\prod_{i \in I}^F \mo{M}_i = (W, {\leq,} \{ \mo{C}_w \}_{w \in W})$.

  \bigskip\noindent
  \textbf{Step 1: defining $W$.}
    Let $W$ be the reduced product $\prod_{i \in I}^F W_i$. In other words, $W$ consists
    of elements in $\prod_{i \in I} W_i$ modulo the equivalence relation $\sim$ given by
    $\alpha \sim \beta \iff \{ i \in I \mid \alpha(i) = \beta(i) \} \in F$.
    Elements in~$W$ are denoted by $\alpha_F$.

  \bigskip\noindent
  \textbf{Step 2: defining $\leq$.}
    Define the relation $\leq$ on $W$ by
    $\alpha_F \leq \beta_F$ iff $\{ i \in I \mid \alpha(i) \leq_i \beta(i) \} \in F$.

    \begin{lemma}\label{lem:B1} \
      \begin{enumerate}
        \item The definition of $\leq$ does not depend on the choice of 
              representative of $\alpha_F$ and $\beta_F$.
        \item The relation $\leq$ defines a partial order on $W$.
      \end{enumerate}
    \end{lemma}
    \begin{proof}
      (1) \; Suppose $\alpha \sim \alpha'$ and $\beta \sim \beta'$.
      Then $\{ i \in I \mid \alpha(i) = \alpha'(i) \} \in F$ and
      $\{ i \in I \mid \beta(i) = \beta'(i) \} \in F$. Note that
      \begin{equation*}
        \{ i \in I \mid \alpha(i) = \alpha'(i) \}
          \cap \{ i \in I \mid \beta(i) = \beta'(i) \}
          \cap \{ i \in I \mid \alpha(i) \leq_i \beta(i) \}
          \subseteq \{ i \in I \mid \alpha'(i) \leq_i \beta'(i) \}.
      \end{equation*}
      Therefore, if $\{ i \in I \mid \alpha(i) \leq_i \beta(i) \} \in F$, then
      the fact that $F$ is a filter, hence closed under intersections
      and upwards closed, entails
      $\{ i \in I \mid \alpha'(i) \leq_i \beta'(i) \} \in F$.
      This proves that $\leq$ is well defined.

      (2) \; We need to verify that $\leq$ is reflexive,
      antisymmetric and transitive.
      Reflexivity follows from the fact that $\{ i \in I \mid \alpha(i) \leq_i \alpha(i) \} = I \in F$.
      For antisymmetry, suppose $\alpha_F \leq \beta_F$ and
      $\beta_F \leq \alpha_F$. Then
      \begin{equation*}
        \{ i \in I \mid \alpha(i) \leq_i \beta(i) \}
          \cap \{ i \in I \mid \beta(i) \leq_i \alpha(i) \}
          = \{ i \in I \mid \alpha(i) = \beta(i) \} \in F,
      \end{equation*}
      so $\alpha \sim \beta$, hence $\alpha_F = \beta_F$.
      Lastly, suppose $\alpha_F \leq \beta_F$ and $\beta_F \leq \gamma_F$.
      Then
      \begin{equation*}
        \{ i \in I \mid \alpha(i) \leq_i \beta(i) \}
          \cap \{ i \in I \mid \beta(i) \leq_i \gamma(i) \}
          \subseteq \{ i \in I \mid \alpha(i) \leq_i \gamma(i) \} \in F,
      \end{equation*}
      so $\alpha_F \leq \gamma_F$.
    \end{proof}

  \smallskip\noindent
  \textbf{Step 3: defining the domains $D_{\alpha_F}$.}
    For each $i \in I$, let $D_i \coloneq D_{\mo{M}_i} = \bigcup_{w \in W_i} D_{i,w}$ be the union
    of the domains $D_{i,w}$ of the $\CL$-structures $\mo{C}_{i,w}$.
    Define $D \coloneq \prod_{i \in I}^F D_i$ to be the reduced
    product of the $D_i$, that is, $D$ is equal to the product $\prod_{i \in I} D_i$ modulo the
    equivalence relation $\approx$ given by $\xi \approx \eta$ iff $\{ i \in I \mid \xi(i) = \eta(i) \} \in F$. Now we define the domain at $\alpha_F \in W$ by
    \begin{equation*}
      D_{\alpha_F} \coloneq \{ \xi_F \in D \mid \{ i \in I \mid \xi(i) \in D_{i,\alpha(i)} \} \in F \}.
    \end{equation*}

    \begin{lemma}\label{lem:D} \
      \begin{enumerate}
        \item The definition of $D_{\alpha_F}$ does not depend on the 
              choice of representative of $\alpha_F$ or $\xi_F$.
        \item If $\alpha_F \leq \beta_F$, then $D_{\alpha_F} \subseteq D_{\beta_F}$.
        \item $D = \bigcup_{\alpha_F \in W} D_{\alpha_F}$
      \end{enumerate}
    \end{lemma}
    \begin{proof}
    Suppose $\alpha \sim \alpha'$ and $\xi \approx \xi'$, so that
      $\{ i \in I \mid \alpha(i) = \alpha'(i) \} \in F$ and $\{ i \in I \mid \xi(i) = \xi'(i) \} \in F$.
      We need to show that $\{ i \in I \mid \xi(i) \in D_{i,\alpha(i)} \} \in F$
      if and only if $\{ i \in I \mid \xi'(i) \in D_{i,\alpha'(i)} \} \in F$.
      Suppose the former is in $F$. Then for any~$j$ in the intersection
      \begin{equation}\label{eq:D-well-def}
        \{ i \in I \mid \alpha(i) = \alpha'(i) \}
          \cap \{ i \in I \mid \xi(i) \in D_{i,\alpha(i)} \} \cap \{ i \in I \mid \xi(i) = \xi'(i) \}
      \end{equation}
      we have $\alpha(j) = \alpha'(j)$ and $\xi(j) = \xi'(j)$
      and $\xi(j) \in D_{j, \alpha(j)}$, which clearly entails
      $\xi'(j) \in D_{j, \alpha'(j)}$. Therefore the intersection of
      \eqref{eq:D-well-def} is contained in
      $\{ i \in I \mid \xi'(i) \in D_{i, \alpha'(i)}\}$.
      Since filters are upwards closed and closed under finite intersections,
      we find $\{ i \in I \mid \xi'(i) \in D_{i,\alpha'(i)}\} \in F$,
      as desired. The other direction of the ``iff'' is analogous.

      For the second item, suppose $\alpha_F \leq \beta_F$ and $\xi_F \in D_{\alpha_F}$.
      Then $\{ i \in I \mid \alpha(i) \leq_i \beta(i) \} \in F$
      and $\{ i \in I \mid \xi(i) \in D_{i,\alpha(i)} \} \in F$.
      Now we note that
      \begin{equation*}
        \{ i \in I \mid \alpha(i) \leq_i \beta(i) \}
          \cap \{ i \in I \mid \xi(i) \in D_{i,\alpha(i)} \}
          \subseteq \{ i \in I \mid \xi(i) \in D_{i,\beta(i)} \},
      \end{equation*}
      because $\alpha(i) \leq_i \beta(i)$ implies $D_{i,\alpha(i)} \subseteq D_{i,\beta(i)}$. 
      Since $F$ is a filter we find
      $\{ i \in I \mid \xi(i) \in D_{i,\beta(i)} \} \in F$,
      hence $\xi_F \in D_{\beta_F}$, so that $D_{\alpha_F} \subseteq D_{\beta_F}$.

      For the third item, $\bigcup_{\alpha_F \in W} D_{\alpha_F} \subseteq D$ follows from the corresponding definition. For the converse, take $\xi_F \in D$. Thus, $\xi: I \to \bigcup_{i \in I} D_i$ is such that $\xi(i) \in D_i=\bigcup_{w \in W_i} D_{i,w}$ for each $i \in I$ and hence, for each $i \in I$, there exists a $w_i \in W_i$ such that $\xi(i) \in D_{i,w_i}$. Define $\alpha_\xi : I \to \bigcup_{i \in I} W_i$ by $\alpha_\xi (i) = w_i \in W_i$. Thus,
      $$\{ i \in I \mid \xi(i) \in D_{i,\alpha_\xi(i)}\} = I$$
      and hence $\xi_F \in D_{{\alpha_\xi}_F}$.
    \end{proof}

  \smallskip\noindent
  \textbf{Step 4: defining the constants.}
  Let $c \in \Con$ be a constant symbol.
  For each $i \in I$, let $c_i$ be the interpretation of $c$
  in $D_i$, and define $\tilde{c} \in \prod_{i \in I} D_i$ by
  \begin{equation*}
    \tilde{c} : I \to \bigcup_{i \in I} D_i : i \mapsto c_i.
  \end{equation*}
  Then, if $\tilde{c}_F \in D_{\alpha_F}$, we let $c \in \Con_{\alpha_F}$ and
  define $\II_{\alpha_F}(c) \coloneq \tilde{c}_F$; otherwise, if $\tilde{c}_F \notin D_{\alpha_F}$, we let $c \notin \Con_{\alpha_F}$
  and we leave $\II_{\alpha_F}(c)$ undefined.
  The fact that $D_{\alpha_F}$ does not depend on the representative
  of $\alpha_F$ entails that this is well defined, and by definition it
  satisfies the requirements from Definition~\ref{def:IL-structure}.

  \bigskip\noindent
  \textbf{Step 5: defining the predicates.}
    Finally, we define the interpretations of the predicates
    on the domains $D_{\alpha_F}$.
    For an $n$-ary predicate $P \in \Pred$, set
    \begin{equation*}
      (\xi_F^1, \ldots, \xi_F^n) \in \II_{\alpha_F}(P)
        \iff \{ i \in I \mid (\xi^1(i), \ldots, \xi^n(i)) \in \II_{i, \alpha(i)}(P) \} \in F.
    \end{equation*}
    Then for each $\alpha_F \in W$ we get a $\CL$-structure
    $\mo{C}_{\alpha_F} = (D_{\alpha_F}, \II_{\alpha_F})$.

\begin{lemma}\label{lem:B3} \
  \begin{enumerate}
    \item $\II_{\alpha_F}(P) \subseteq D^n_{\alpha_F}$.
    \item The definition of $\II_{\alpha_F}(P)$ does not depend on the
          choice of $\alpha_F$ or any of the $\xi^j_F$.
    \item If $\alpha_F \leq \beta_F$, then
          $\II_{\alpha_F}(P) \subseteq \II_{\beta_F}(P)$.
  \end{enumerate}
\end{lemma}
\begin{proof}
  (1) \; Suppose $(\xi_F^1, \ldots, \xi_F^n) \in \II_{\alpha_F}(P)$.
  We aim to show $\xi_F^j \in D_{\alpha_F}$ for each $j \in \{ 1, \ldots, n \}$.
  By assumption,
  $\{ i \in I \mid (\xi^1(i), \ldots, \xi^n(i)) \in \II_{i, \alpha(i)}(P) \} \in F$.
  Since $\II_{i, \alpha(i)}(P) \subseteq D^n_{i,\alpha(i)}$ for each $i \in I$,
  this implies
  \begin{equation*}
    \{ i \in I \mid (\xi^1(i), \ldots, \xi^n(i)) \in \II_{i, \alpha(i)}(P) \}
      \subseteq  \{ i \in I \mid \xi^j(i) \in D_{i,\alpha(i)}\},
  \end{equation*}
  so that $\{ i \in I \mid \xi^j(i) \in D_{i, \alpha(i)} \} \in F$ for each
  $j \in \{ 1, \ldots, n \}$. By definition, this implies $\xi^j_F \in D_{\alpha_F}$.

  (2) \; Suppose $\alpha \sim \beta$ and $\xi^j \approx \eta^j$ for
  all $j \in \{ 1, \ldots, n \}$.
  Then $\{ i \in I \mid \alpha(i) = \beta(i) \} \in F$ and
  $\{ i \in I \mid \xi^j(i) = \eta^j(i) \} \in F$ for all $j \in \{ 1, \ldots, n \}$.
  We need to show that
  \begin{equation*}
    \{i \in I \mid (\xi^1(i), \ldots, \xi^n(i)) \in \II_{i,\alpha(i)}(P) \} \in F
    \iff
    \{i \in I \mid (\eta^1(i), \ldots, \eta^n(i)) \in \II_{i,\beta(i)}(P) \} \in F.
  \end{equation*}
  The left-to-right direction follows from the fact that
  \begin{align*}
    \{i \in I &\mid (\xi^1(i), \ldots, \xi^n(i)) \in \II_{i,\alpha(i)}(P) \}
      \cap \{ i \in I \mid \alpha(i) = \beta(i) \} \\
      &\cap \{ i \in I \mid \xi^1(i) = \eta^1(i) \} 
      \cap \cdots \cap \{ i \in I \mid \xi^n(i) = \eta^n(i) \}
      \subseteq \{i \in I \mid (\eta^1(i), \ldots, \eta^n(i)) \in \II_{i,\beta(i)}(P) \}.
  \end{align*}
  The other direction can be proven analogously.
  (3) is similar to the proof of Lemma~\ref{lem:D}.
\end{proof}

  Suppose $\rho_i$ is an assignment for $\mo{M}_i$, for each $i \in I$.
  For each $x \in \Var$, define
  $\rho(x) : I \to \bigcup_{i \in I} D_i : i \mapsto \rho_i(x)$.
  Then the \emph{product assignment} $\rho_F$
  is given by letting $\rho_F(x)$ be the equivalence class of $\rho(x)$,
  i.e.~$\rho_F(x) \coloneq \rho(x)_F$.
  Conversely, every assignment $\rho_F$ for $\prod_{i \in I}^F\mo{M}_i$ can be obtained
  in this way from assignments $\rho_i$ given by $\rho_i(x) = \rho(x)(i)$.

\begin{theorem}
  Let $I$ be a set, and for each $i \in I$ let $\mo{M}_i$ be a
  first-order structure and $\rho_i$ an assignment for $\mo{M}_i$.
  Let $F$ be an ultrafilter on $I$,
  $\mo{M} \coloneq \prod_{i \in I}^F \mo{M}_i$ the filter product,
  and $\rho_F$ the product assignment for~$\mo{M}$.
  Then for any world $\alpha_F \in W$ and any formula $\phi(x_1, \ldots, x_n)$
  such that $\rho_F(x_1), \ldots, \rho_F(x_n) \in D_{\alpha_F}$,
  \begin{equation*}
    \satisfy{\mo{M}}{\alpha_F}{\phi}{\rho_F}
      \iff \{ i \in I \mid \satisfy{\mo{M}_i}{\alpha(i)}{\phi}{\rho_i} \} \in F.
  \end{equation*}
\end{theorem}
\begin{proof}[Proof of Theorem~\ref{thm:Los}]
  We use induction on the structure of $\phi$.

  \medskip\noindent
  \textit{$\triangleright\;$ Case for $\phi = P(t_1, \ldots, t_m)$.}
    Compute
    \begin{align*}
      \satisfy{\mo{M}}{\alpha_F}{P(t_1, \ldots, t_m)}{\rho_F}
        &\iff (\rho^{\mo{M}}_F(t_1), \ldots, \rho^{\mo{M}}_F(t_m)) \in \II_{\alpha_F}(P)
        &\text{(definition of $\satisfySym^{\rho_F}$)} \\
        &\iff \{ i \in I \mid (\rho^{\mo{M}}(t_1)(i), \ldots, \rho^{\mo{M}}(t_m)(i)) \in \II_{i, \alpha(i)} \} \in F 
        &\text{(definition of $\II_{i, \alpha(i)}$)} \\
        &\iff \{ i \in I \mid (\rho^{\mo{M}_i}_i(t_1), \ldots, \rho^{\mo{M}_i}_i(t_m)) \in \II_{i, \alpha(i)} \} \in F
        &\text{(definition of $\rho_F$)} \\
        &\iff \{ i \in I \mid \satisfy{\mo{M}_i}{\alpha(i)}{P(t_1, \ldots, t_m)}{\rho_i} \} \in F
        &\text{(definition of $\satisfySym^{\rho_i}$)}
    \end{align*}

  \medskip\noindent
  \textit{$\triangleright\;$ Case for $\phi = (\psi \wedge \chi)$.}
    If $\satisfy{\mo{M}}{\alpha_F}{\psi \wedge \chi}{\rho_F}$,
    then $\satisfy{\mo{M}}{\alpha_F}{\psi}{\rho_F}$ and
    $\satisfy{\mo{M}}{\alpha_F}{\chi}{\rho_F}$.
    By induction, $\{ i \in I \mid \satisfy{\mo{M}_i}{\alpha(i)}{\psi}{\rho_i} \} \in F$
    and $\{ i \in I \mid \satisfy{\mo{M}_i}{\alpha(i)}{\chi}{\rho_i} \} \in F$,
    and using the fact that $F$ is a filter we find
    \begin{equation*}
      \{ i \in I \mid \satisfy{\mo{M}_i}{\alpha(i)}{\psi}{\rho_i} \}
        \cap \{ i \in I \mid \satisfy{\mo{M}_i}{\alpha(i)}{\chi}{\rho_i} \}
        \subseteq \{ i \in I \mid \satisfy{\mo{M}_i}{\alpha(i)}{\psi \wedge \chi}{\rho_i} \} \in F.
    \end{equation*}
    
    Conversely, if $\{ i \in I \mid \satisfy{\mo{M}_i}{\alpha(i)}{\psi \wedge \chi}{\rho_i} \}\in F$,
    we may use the fact that $F$ is a filter to conclude that
    \begin{equation*}
      \{ i \in I \mid \satisfy{\mo{M}_i}{\alpha(i)}{\psi \wedge \chi}{\rho_i} \}
        \subseteq \{ i \in I \mid \satisfy{\mo{M}_i}{\alpha(i)}{\psi}{\rho_i} \} \in F,
    \end{equation*}
    and, similarly, we conclude that $\{ i \in I \mid \satisfy{\mo{M}_i}{\alpha(i)}{\chi}{\rho_i} \} \in F$.
    The induction hypothesis then entails $\satisfy{\mo{M}}{\alpha_F}{\psi}{\rho_F}$ and
    $\satisfy{\mo{M}}{\alpha_F}{\chi}{\rho_F}$, from which it follows that
    $\satisfy{\mo{M}}{\alpha_F}{\psi \wedge \chi}{\rho_F}$.

  \medskip\noindent
  \textit{$\triangleright\;$ Case for $\phi = (\psi \vee \chi)$.}
    This is similar to the previous case, but using the fact that $F$ is prime.

  \medskip\noindent
  \textit{$\triangleright\;$ Case for $\phi = (\psi \to \chi)$.}
    Suppose $\{ i \in I \mid \satisfy{\mo{M}_i}{\alpha(i)}{\psi \to \chi}{\rho_i} \} \in F$.
    We show that $\satisfy{\mo{M}}{\alpha_F}{\psi \to \chi}{\rho_F}$.
    To this end, let $\alpha_F \leq \beta_F$ and suppose
    $\satisfy{\mo{M}}{\beta_F}{\psi}{\rho_F}$.
    Then, using the induction hypothesis, we obtain
    \begin{equation*}
      \{ i \in I \mid \alpha(i) \leq_i \beta(i) \} \in F
      \quad\text{and}\quad
      \{ i \in I \mid \satisfy{\mo{M}_i}{\beta(i)}{\psi}{\rho_i} \} \in F.
    \end{equation*}
    Combining this with the assumption gives
    \begin{align*}
      \{ i \in I \mid \alpha(i) \leq_i \beta(i) \}
      &\cap \{ i \in I \mid \satisfy{\mo{M}_i}{\beta(i)}{\psi}{\rho_i} \} \\
      &\cap \{ i \in I \mid \satisfy{\mo{M}_i}{\alpha(i)}{\psi \to \chi}{\rho_i} \}
      \subseteq \{ i \in I \mid \satisfy{\mo{M}_i}{\beta(i)}{\chi}{\rho_i} \} \in F,
    \end{align*}
    so by the induction hypothesis again we obtain
    $\satisfy{\mo{M}}{\beta_F}{\chi}{\rho_F}$.
    Therefore $\satisfy{\mo{M}}{\alpha_F}{\psi \to \chi}{\rho_F}$.

    For the converse, suppose 
    \begin{align*}
      \{ i \in I \mid \satisfy{\mo{M}_i}{\alpha(i)}{\psi\to\chi}{\rho_i} \} \notin F
    \end{align*}
    Since $F$ is an ultrafilter, its complement is an element of $F$:
    \begin{align*}
      S \coloneq \{ i \in I \mid \notsatisfy{\mo{M}_i}{\alpha(i)}{\psi\to\chi}{\rho_i} \} \in F
    \end{align*}
    Hence, for every $i \in S$, there exists a $b_i \in W_i$ such that
    $\alpha(i) \leq_i b_i$ and $\satisfy{\mo{M}_i}{b_i}{\psi}{\rho_i}$
    and $\notsatisfy{\mo{M}_i}{b_i}{\chi}{\rho_i}$.
    Define $\beta \in \prod_{i \in I} W_i$ by
    \begin{align*}
      \beta(i)
        = \begin{cases}
            b_i &\text{if } i \in S \\
            \alpha(i) &\text{if } i \notin S
          \end{cases}
    \end{align*}
    Then $\beta_F \in W$ and $\alpha_F \leq_F \beta_F$
    because $S \subseteq \{ i \in I \mid \alpha(i) \leq_i \beta(i) \}$ and $S \in F$.
    Furthermore,
    $S \subseteq \{ i \in I \mid \satisfy{\mo{M}_i}{\beta(i)}{\psi}{\rho_i} \}$
    and $S \subseteq \{ i \in I \mid \notsatisfy{\mo{M}_i}{\beta(i)}{\chi}{\rho_i} \}$, 
    so
    \begin{equation*}
      \{ i \in I \mid \satisfy{\mo{M}_i}{\beta(i)}{\psi}{\rho_i} \} \in F
      \quad\text{and}\quad
      \{ i \in I \mid \notsatisfy{\mo{M}_i}{\beta(i)}{\chi}{\rho_i} \} \in F,
    \end{equation*}
    hence $\{ i \in I \mid \satisfy{\mo{M}_i}{\beta(i)}{\chi}{\rho_i} \} \notin F$.
    The induction hypothesis then yields $\satisfy{\mo{M}}{\beta_F}{\psi}{\rho_F}$
    and $\notsatisfy{\mo{M}}{\beta_F}{\chi}{\rho_F}$.
    Since $\alpha_F \leq_F \beta_F$, this proves
    $\notsatisfy{\mo{M}}{\alpha_F}{\psi\to\chi}{\rho_F}$.

  \medskip\noindent
  \textit{$\triangleright\;$ Case for $\phi = \forall x\,\psi$.}
    Suppose that
    $\{ i \in I \mid \satisfy{\mo{M}_i}{\alpha(i)}{\forall x\,\psi}{\rho_i} \} \notin F$. 
    Since $F$ is an ultrafilter, 
    \begin{equation*}
      S \coloneq \{ i \in I \mid \notsatisfy{\mo{M}_i}{\alpha(i)}{\forall x\,\psi}{\rho_i}\} \in F
    \end{equation*}
    and hence for every $i \in S$ there exist a world $b_i \in W_i$ and a domain
    element $d_i \in D_{i,b_i}$ such that $\alpha(i) \leq_i b_i$
    and $\notsatisfy{\mo{M}_i}{b_i}{\psi}{\rho_i[x \coloneq d_i]}$.
    Define $\beta \in \prod_{i \in I} W_i$ by
    \begin{align*}
      \beta(i)
        = \begin{cases}
            b_i &\text{if } i \in S \\
            \alpha(i) &\text{if } i \notin S
          \end{cases}
    \end{align*}
    and $\xi^d \in \prod_{i \in I} D_i$ by
    \begin{align*}
      \xi^d(i)
        = \begin{cases}
            d_i   &\text{if } i \in S \\
            d^*_i &\text{if } i \notin S, \text{where } d^*_i \text{ is any element of } D_{i,\beta(i)}
          \end{cases}
    \end{align*}
    Then $\beta_F \in W$ and $\xi^d_F \in D_{\beta_F}$, because
    $S \subseteq \{ i \in I \mid \xi^d(i) \in D_{i, \alpha(i)} \} \in F$.
    Since $S \subseteq \{ i \in I \mid \alpha(i) \leq_i \beta(i) \} \in F$
    we have $\alpha_F \leq_F \beta_F$.
    Furthermore, by construction we have 
    \begin{equation*}
      S \subseteq \{ i \in I \mid \notsatisfy{\mo{M}_i}{\beta(i)}{\psi}{\rho_i[x \coloneq \xi^d(i)]} \} \in F
    \end{equation*}
    so $\{ i \in I \mid \satisfy{\mo{M}_i}{\beta(i)}{\psi}{\rho_i[x \coloneq \xi^d(i)]} \} \not\in F$.
    Using the fact that $\rho_F[x \coloneq \xi^d_F]$ is the product of
    the assignments $\rho_i[x \coloneq \xi^d(i)]$ and the induction hypothesis
    we obtain $\notsatisfy{\mo{M}}{\beta_F}{\psi}{\rho_F[x \coloneq \xi^d_F]}$.
    Therefore $\notsatisfy{\mo{M}}{\alpha_F}{\forall x\, \psi}{\rho_F}$.

    For the converse, suppose that $\notsatisfy{\mo{M}}{\alpha_F}{\forall x\, \psi}{\rho_F}$. 
    Then there exist a world $\beta_F \in W$ and an individual $\xi_F \in D_{\beta_F}$
    such that $\alpha_F \leq_F \beta_F$
    and $\notsatisfy{\mo{M}}{\beta_F}{\psi}{\rho_F[x \coloneq \xi_F]}$. 
    This implies $\{ i \in I \mid \alpha(i) \leq_i \beta(i) \} \in F$
    and $\{ i \in I \mid \xi(i) \in D_{i,\beta(i)}\} \in F$
    and (by the induction hypothesis)
    $\{ i \in I \mid \satisfy{\mo{M}_i}{\beta(i)}{\psi}{\rho_i[x \coloneq \xi(i)]} \} \notin F$.
    Since $F$ is an ultrafilter, we get
    $\{ i \in I \mid \notsatisfy{\mo{M}_i}{\beta(i)}{\psi}{\rho_i[x \coloneq \xi(i)]} \} \in F$.
    It follows that
    \begin{align*}
      \{ i \in I \mid \alpha(i) \leq_i \beta(i) \}
        &\cap \{ i \in I \mid \xi(i) \in D_{i,\beta(i)} \} \\
        &\cap \{ i \in I \mid \notsatisfy{\mo{M}_i}{\beta(i)}{\psi}{\rho_i[x \coloneq \xi(i)]} \} 
         \subseteq \{ i \in I \mid \notsatisfy{\mo{M}_i}{\alpha(i)}{\forall x\,\psi}{\rho_i}\} \in F
    \end{align*}
    and hence $\{ i \in I \mid \satisfy{\mo{M}_i}{\alpha(i)}{\forall x\,\psi}{\rho_i} \} \notin F$.
    
  \medskip\noindent
  \textit{$\triangleright\;$ Case for $\phi = \exists x\,\psi$.}
    Suppose $\satisfy{\mo{M}}{\alpha_F}{\exists x\,\psi}{\rho_F}$.
    Then there exists some individual $\xi_F \in D_{\alpha_F}$ such that
    $\satisfy{\mo{M}}{\alpha_F}{\psi}{\rho_F[x \coloneq \xi_F]}$.
    By the induction hypothesis, and using the fact that
    $\rho_F[x \coloneq \xi_F]$ is the product of the assignments $\rho_i[x \coloneq \xi(i)]$,
    we obtain
    $\{ i \in I \mid \satisfy{\mo{M}_i}{\alpha(i)}{\psi}{\rho_i[x \coloneq \xi(i)]} \} \in F$.
    Clearly
    \begin{align*}
      \{ i \in I \mid \satisfy{\mo{M}_i}{\alpha(i)}{\psi}{\rho_i[x \coloneq \xi(i)]}\}
        \subseteq \{ i \in I \mid \satisfy{\mo{M}_i}{\alpha(i)}{\exists x\,\psi}{\rho_i} \},
    \end{align*}
    so $\{ i \in I \mid \satisfy{\mo{M}_i}{\alpha(i)}{\exists x\, \psi}{\rho_i} \} \in F$.
  
    Now suppose
    $S \coloneq \{ i \in I \mid \satisfy{\mo{M}_i}{\alpha(i)}{\exists x\,\psi}{\rho_i}\} \in F$. 
    Then, for every $i \in S$ there exists a $d_i \in D_{i,\alpha(i)}$
    such that $\satisfy{\mo{M}_i}{\alpha(i)}{\psi}{\rho_i[x \coloneq d_i]}$.
    Define $\xi^d \in \prod_{i \in I} D_i$ by
    \begin{equation*}
      \xi^d(i)
        = \begin{cases}
            d_i   &\text{if } i \in S \\
            d^*_i &\text{if } i \not\in S, \text{where } d^*_i \text{ is any element of } D_{i,\alpha(i)}
          \end{cases}
    \end{equation*}
    Then $\xi^d_F \in D_{\alpha_F}$ and
    \begin{equation*}
      S \subseteq \{ i \in I \mid \satisfy{\mo{M}_i}{\alpha(i)}{\psi}{\rho_i[x \coloneq \xi^d(i)]} \},
    \end{equation*}
    so $\{ i \in I \mid \satisfy{\mo{M}_i}{\alpha(i)}{\psi}{\rho_i[x \coloneq \xi^d(i)]} \} \in F$
    and by the induction hypothesis $\satisfy{\mo{M}}{\alpha_F}{\psi}{\rho_F[x \coloneq \xi^d_F]}$.
\end{proof}

\bibliographystyle{eptcs}
\bibliography{biblio}

\end{document}